\documentclass[reqno, 11pt]{amsart}

\usepackage{amssymb}
\usepackage{amsfonts}
\usepackage{amsmath}
\usepackage{amsthm} 
\usepackage{xcolor}
\usepackage{graphicx}
\usepackage{multicol}
\usepackage{enumitem}
\setlist[itemize]{leftmargin=6mm} 

\usepackage{mathptmx}
\usepackage[mathscr]{eucal}

\usepackage{wrapfig}
\usepackage[labelformat=empty]{subcaption}

\usepackage[margin=2.5cm]{geometry}

\usepackage{setspace}
\usepackage{comment}

\usepackage[colorlinks=true,citecolor=green,urlcolor=blue,linkcolor=blue]{hyperref}

\numberwithin{equation}{section}

\theoremstyle{plain}
\newtheorem{theorem}{Theorem}[section]
\newtheorem{proposition}[theorem]{Proposition} 
\newtheorem{lemma}[theorem]{Lemma} 
\newtheorem{corollary}[theorem]{Corollary} 
\newtheorem{conjecture}[theorem]{Conjecture}

\theoremstyle{definition}
\newtheorem{definition}[theorem]{Definition} 
\newtheorem{remark}[theorem]{Remark} 
\newtheorem{example}[theorem]{Example} 

\newcommand{\norm}{\varphi}
\newcommand{\p}{\partial}

\newcommand{\cl}[1]{\overline{#1}}

\newcommand{\scalarp}[2]{\left\langle #1,\, #2\right\rangle}
\newcommand{\strictlyincluded}{\Subset}
\newcommand{\subgraph}{{\rm sg}}
\newcommand{\epigraph}{{\rm epi}}
\newcommand{\supp}{\mathrm{supp\,}}
\newcommand{\dist}{{\rm dist}}

\newcommand{\res}{\mathop{\hbox{\vrule height 7pt width 0.5pt depth 0pt \vrule height 0.5pt width 3pt depth 0pt}}\nolimits}
\newcommand{\Wulff}{W}
\newcommand{\oWulff}{\mathring{W}}

\newcommand{\loc}{{\mathrm{loc}}}
\newcommand{\Lip}{{\mathrm{Lip}}}
\renewcommand{\div}{\mathrm{div}}

\newcommand{\R}{\mathbb{R}}
\renewcommand{\S}{\mathbb{S}}

\newcommand{\cA}{\mathcal{A}}
\newcommand{\cP}{\mathcal{P}}
\newcommand{\cV}{\mathcal{V}}
\newcommand{\cO}{\mathcal{O}}

\newcommand{\cF}{\mathcal{F}}
\newcommand{\cG}{\mathcal{G}}

\newcommand{\sL}{\mathscr{L}}
\newcommand{\sH}{\mathscr{H}}

\newcommand{\be}{\mathbf{e}}

\setlist[itemize]{leftmargin=6mm} 

\title{A De Giorgi conjecture on the regularity of minimizers of Cartesian area in 1D}

\author[G. Bellettini]{Giovanni Bellettini} 
\address[G.Bellettini]{University of Siena, Via Roma 56, 53100 Siena, Italy \& International Centre for Theoretical Physics (ICTP), Strada Costiera 11, 34151 Trieste, Italy}
\email{giovanni.bellettini@unisi.it}

\author[Sh. Kholmatov]{Shokhrukh Yu. Kholmatov}
\address[Sh. Kholmatov]{University of Vienna,Oskar-Morgenstern-Platz 1, 1090  Vienna, Austria \& 
Samarkand State University, University boulevard 15, 140104 Samarkand, Uzbekistan}
\email{shokhrukh.kholmatov@univie.ac.at}

\subjclass[2010]{49Q15, 49Q20, 26A06, 26A45}

\keywords{$L^p$-perturbation, anisotropic area functional, Cartesian area, BV-function, De Giorgi conjecture}

\date{\today}
\begin{document}

\begin{abstract}
We prove a $C^{1,1}$-regularity of minimizers of the functional 
$$
\int_I \sqrt{1+|Du|^2} + \int_I |u-g|ds,\quad u\in BV(I),
$$
provided $I\subset\R$ is a bounded open interval and $\|g\|_\infty$ is sufficiently small, thus partially establishing a De Giorgi conjecture in dimension one and codimension one. We also extend our result to a suitable anisotropic setting.
\end{abstract}

\maketitle

\section{Introduction}

The non-parametric minimal surfaces, more generally, the  prescribed mean curvature surfaces, have been extensively studied in the literature from the variational perspective (see e.g. \cite{Giusti:1980,Massari:1974, MM:1984_book, Miranda:1964, LS:2018} and the references therein). 
Given an open set $\Omega\subset\R^n$ and a sufficiently regular function $H:\Omega\to\R,$ the underlying equation is rewritten as
\begin{equation}\label{poplanuyc}
\div \frac{\nabla u}{\sqrt{1+|\nabla u|^2}} = H\quad\text{in $\Omega$}
\end{equation}
with a prescribed Dirichlet or Neumann boundary condition, and corresponds to the Euler-Lagrange equation of the functional
\begin{equation}\label{azstcub112}
\int_\Omega\sqrt{1+|\nabla u|^2}dx + \int_\Omega Hu\,dx,\quad u\in C^1(\Omega).
\end{equation}
It is well-known that under suitable assumptions on $\Omega$ and $H,$ the minimizers are in fact locally $C^{2+\alpha},$ and hence solve \eqref{poplanuyc} in a classical sense.

In the context of functionals with linear growth, a related problem is the existence and regularity of minimizers of the (convex, but not strictly convex) functional 
$$
\cF(u):=\int_\Omega\sqrt{1+|\nabla u|^2}dx + \int_\Omega |u-g|\,dx,\quad u\in C^1(\Omega),
$$
where $g\in L^1(\Omega)$ is given, see \cite{Degiorgi:1992}. In this case, the associated Euler-Lagrange equation becomes formally a differential inclusion of the form
\begin{equation}\label{minsur_equation}
\div \frac{\nabla u}{\sqrt{1+|\nabla u|^2}} \in 
\begin{cases}
\{1\} & \text{in $\{u>g\},$}\\
[-1,1] & \text{in $\{u=g\},$} \\
\{-1\} & \text{in $\{u<g\},$}
\end{cases}
\end{equation}
thus, in the sets $\{u>g\}$ and $\{u<g\},$ the subgraph of $u$ has mean curvature equal to $1$ and $-1,$ respectively. 

Unlike the minimizers of the functional in \eqref{azstcub112},  the equation \eqref{minsur_equation} may admit nonregular solutions, as observed in \cite{Degiorgi:1992}. For instance, if $n=1,$ $\Omega=(-1,1)$ and 
$$
g(s) 
= 
\begin{cases}
2 & \text{if $s\in(0,1),$}\\
-2 & \text{if $s\in(-1,0),$}
\end{cases}
$$
one can readily check that the functions
\begin{equation}\label{ahst6bnega}
u_{a,b}(s)=
\begin{cases}
\sqrt{2s-s^2} + a & \text{if $s\in(0,1),$}\\
0 & \text{if $s=0,$} \\
\sqrt{-2s-s^2} + b & \text{if $s\in(-1,0)$}
\end{cases}
\end{equation}
with $-1\le b\le a\le 1$ satisfy \eqref{minsur_equation} and minimize $\cF,$ with $\cF(u_{a,b})=4+\frac{\pi}{2}.$ 
However, any $u_{a,b}$ is not continuously differentiable at $s=0,$ even worse -- it has a jump if $a>b.$ 

To study regularity of minimizers of $\cF,$ in \cite{Degiorgi:1992} De Giorgi posed the following conjecture, which seems nontrivial even when $n=k=1.$

\begin{conjecture}\label{conj:dg_area}
For any $n,k\ge1,$ there exists $\sigma:=\sigma(n,k)>0$ such that for any open ball $B\subset\R^n$  and $g\in L^\infty(B;\R^k)$ with $\|g\|_\infty\le \sigma$ the following minimum is achieved:
\begin{equation}\label{min_problem_area}
\min \Big\{\int_B \sqrt{1+\sum |M_i(\nabla u)|^2} dx  + \int_B |u-g|dx:\,\,u\in C^1(B;\R^k)\Big\},
\end{equation}
where the sum is  taken over all minors $M_i(\nabla u)$  of the  Jacobian matrix $\nabla u $ of $u.$
\end{conjecture}

A variation of this conjecture for $n=1$ and $k\ge1$ has been recently addressed in \cite{CLSS:2025}: using the Sobolev regularity theory for the minimizers of an Ambrosio-Tortorelli-type functional \cite{AT:1990_cpam}, the authors have shown the existence of $\sigma:=\sigma(k,|I|^{-1/2})>0,$ such that for any $g\in L^\infty(I;\R^k)$ with $\|g\|_\infty\le \sigma,$ the minimum  problem
\begin{equation}\label{min_problem_area2_scala}
\min\Big\{\int_I \sqrt{1+|u'|^2}\,ds + \int_I (u-g)^2\,ds:\,\, u\in C^1(I;\R^k)\Big\}
\end{equation}
admits a unique solution, where $I$ is a bounded interval. This result does not solve Conjecture \ref{conj:dg_area}, due to the exponent $2$ in the 
second integral of the functional and to the dependence of $\sigma$ on the length $|I|$ of the interval $I.$ To prove the existence of solutions, they observe  that if $u\in \Wulff^{1,\infty}(I;\R^k)$  minimizes the $\Gamma$-limit $F$ of a suitable sequence of approximating functionals,  then it also minimizes the functional in \eqref{min_problem_area2_scala}, which turns out to be Sobolev regular provided that $\|g\|_\infty$ is small enough depending only on $|I|.$ Next, they show that $u$ is in fact a solution to the corresponding Euler-Lagrange equation with suitable boundary conditions, which yield the continuity of the derivative (here the 
quadratic term $(u-g)^2$ is important in the analysis of the Euler-Lagrange equation).

In the present paper we consider $n=k=1$ and generalize the functional in \eqref{min_problem_area} to the anisotropic case with $L^p$-fidelity terms. Given an anisotropy (a norm) $\norm$ in $\R^2,$ $p\in[1,+\infty),$ a bounded open interval $I\subset\R$ and $g\in L^\infty(I),$ we consider the functional  
\begin{equation}\label{G_phi_pgI}
\cG(u) = \int_I\norm^o(-Du,1) + \int_I|u-g|^pds,\quad u\in L^1(I),
\end{equation}
where $\norm^o$ is the dual of $\norm$ and 
$$
\int_I \norm^o(-Du,1):=\sup\Big\{\int_I (uh_1'+h_2)ds:\,\, (h_1,h_2)\in C_c^1(I;\R^2),\,\, \|\norm(h_1,h_2)\|_\infty\le 1\Big\}
$$
is the $\norm$-total variation of $(-Du,\sL^1)$ when $u\in BV(I).$ The main result of this paper reads as follows (see also Theorem \ref{teo:regular_minimizers}).

\begin{theorem}\label{teo:minmi_dg}
Let $\norm$ be an anisotropy in $\R^2$ such that 
the unit ball $\Wulff^\norm:=\{\norm\le1\}$ is symmetric with respect to the coordinate axes and does not have vertical facets. Let $I\subset\R $ be a bounded open interval.  Then there exists $\sigma:=\sigma(\norm,p,|I|)>0$ such that for any $g\in L^\infty(I)$ with $\|g\|_\infty< \sigma$ every minimizer of $\cG$ is Lipschitz in $I.$ Additionally, if $\norm$ is $C^2$ out of the origin and elliptic (see Definition \ref{def:elliptic_aniso}), then minimizers are $C^{1,1}$ in $I.$ 
\end{theorem}

In the Euclidean case $\norm=|\cdot|,$ Theorem \ref{teo:minmi_dg} provides a positive solution to Conjecture  \ref{conj:dg_area} for $n=k=1,$ except that  our $\sigma$ depends on $|I|$ (as in \cite{CLSS:2025}); at the same time we gain an extra regularity of minimizers. 

To prove Theorem \ref{teo:minmi_dg}, we begin by observing that if $g$ is bounded, then  every minimizer $u$ of $\cG$ is also bounded, with $\|u\|_\infty\le \|g\|_\infty$ (see Lemma \ref{lem:exist_unique}). If, additionally, $\norm$ is even in each coordinate  (equivalently, $W^\norm$ is symmetric with respect to the coordinate axes), then the subgraph and the epigraph of $u$ are  $(\gamma,\Lambda)$-local minimizers of the $\norm$-perimeter, in the sense of Definition \ref{def:cond_locmin} below, with  suitable constants $\gamma,\Lambda$ (Proposition \ref{prop:cond_locminimi}). Next in Proposition \ref{prop:uniform_wulff_locmin}, we prove that if $u$ satisfies an appropriate $L^\infty$-bound depending only on $\Lambda$ and $|I|,$ a tangent $\norm$-ball condition  holds at each point $x$ on the graph of $u$ for all radii $r$ up to $\frac{\alpha_0}{\Lambda}>0$ for a constant $\alpha_0>0$ depending only on $\norm.$ 
The uniformity of the radii of these tangent balls allows to estimate the deviation of the  generalized normals of the reduced boundary of the subgraph of $u$ in $I\times \R$ from the vertical direction (see \eqref{normal_change}) which is away from $0$ provided that $\|u\|_\infty$ is small enough depending only on $\norm, $ $\Lambda$ and $|I|.$ In particular, this observation and \cite[Lemma 3.10]{NP:2002_narwa} imply that $u$ is Lipschitz in $I$ (Corollary \ref{cor:lip_regular_funcos}). Finally, an explicit choice of $\sigma$ is made, using the previous $L^\infty$-bounds for $u$  (Theorem \ref{teo:regular_minimizers}). When $\norm$ is $C^2$ and elliptic, then the tangent $\norm$-ball condition becomes  equivalent to the classical tangent ball condition, and hence $u$ must be $C^{1,1}$ in $I.$

Note that the function $u_{a,b}$ in \eqref{ahst6bnega} shows that in case $\|g\|_\infty$ is large, the validity of a tangent ball condition may not suffice for the regularity of minimizers of $\cF.$

When $\Lambda=0$ and $\gamma=+\infty,$ $(\gamma,\Lambda)$-local minimizers coincide with  
classical local minimizers. In this case, assuming $\norm$ is symmetric with respect 
to the 
coordinate axes, we can characterize all possible Cartesian 
local minimizers of the $\norm$-perimeter (see Theorem \ref{teo:0_loc_minos}).

The paper is organized as follows. In Section \ref{sec:preliminary} 
we introduce some preliminaries on 
anisotropies, $\Lambda$-local minimizers, and 
$\norm$-ball condition for Cartesian $\Lambda$-local minimizers. In Section \ref{sec:class_local_mini} we provide a characterization of local minimizers. Some regularity properties of $\Lambda$-local minimizers and 
their further  generalizations are studied in Section \ref{sec:regul_lambda_min}. Finally, we prove Theorem \ref{teo:minmi_dg} in Section \ref{sec:application}.

\subsection*{Acknowledgements}
The first author acknowledges the support of the INDAM/GNAMPA. The second author acknowledges support from the Austrian Science Fund (FWF) Stand-Alone project P 33716. The authors are very grateful to Umberto Massari, whose  papers, as well as his book on minimal surfaces written in collaboration with M. Miranda, have been a source of lasting mathematical inspiration. 

\section{Some preliminaries}\label{sec:preliminary}

In what follows, by $\sL^m$ (typically for $m=1,2$) and $\sH^t$ we denote the Lebesgue measure in $\R^m$ and the $t$-dimensional Hausdorff measure in $\R^2.$ Depending on the context, we use  $|\cdot|$ to denote the Euclidean norm of a vector in $\R^2,$ the length of a bounded interval on $\R$ and the measure of a (Lebesgue) measurable set in $\R^m$ for $m=1,2.$ The scalar product in $\R^2$ is indicated by $\scalarp{\cdot}{\cdot}.$ The 
symmetric difference of sets $A$ and $B$ 
is  denoted 
$A\Delta B.$ The symbol $B_r(x)$ stands for the Euclidean ball in $\R^2$ centered at $x$ and of radius $r>0.$ The topological closure, interior and boundary of $E\subset\R^2$ will be denoted by $\cl{E},$ $\mathring{E}$ and $\p E,$ respectively. Given an open interval $I\subset\R,$ we write $\cO(I)$ and $\cO_b(I)$ to denote the collection of all open and all bounded open subsets of $I,$ respectively.

\subsection{Anisotropies}

Let $\norm:\R^2\to[0,+\infty)$ be an anisotropy, i.e., a positively one-homogeneous even convex function with 
\begin{equation}\label{norm_bounds}
c\le \norm\le \frac{1}{c} \quad\text{on the unit circle $\S^1$}
\end{equation}
for some $c\in(0,1].$
We denote by $\norm^o$ the dual of $\norm,$ defined as
$$
\norm^o(\xi) = \max_{\norm(\eta)=1}\,\scalarp{\xi}{\eta},
$$
which is also an anisotropy in $\R^2.$ We say that $\norm$ is a $C^k$-anisotropy for some $k\ge1$ provided that $\norm\in C_\loc^k(\R^2\setminus \{0\}).$

The unit $\norm$-ball $\Wulff^\norm:=\{\norm\le1\}$ is sometimes called the Wulff shape of $\norm.$ We also introduce  the Wulff shape  of radius $r$ centered at $x$ as $\Wulff_r^\norm(x):=\{\norm(\cdot-x)\le r\};$ clearly $\oWulff_r^\norm(x)=\{\norm(\cdot-x)<r\}.$ Given $\eta \in \p \Wulff^\norm,$ we call any vector $\nu\in \p\norm^o(\eta)$ a \emph{normal} to $\Wulff^\norm$ at $\eta,$ where $\p $ is the subdifferential. Note that if $\Wulff^\norm$ is not regular at $\eta,$ for instance, it has a corner, its set of 
normals at $\eta$ forms a nonempty closed convex cone.

We write 
$\dist_\norm(x,S):=\inf\{\norm(x-y):\,\,y\in S\}$ to denote the $\norm$-distance function from a nonempty set $S.$

\begin{definition}[Elliptic anisotropy]\label{def:elliptic_aniso}
An anisotropy $\norm$ in $\R^2$ is \emph{elliptic} provided that there exists $\lambda>0$ such that $\phi - \lambda|\cdot|$  is also an anisotropy in $\R^2.$
\end{definition}

For instance, any anisotropy induced by some positive definite quadratic form, is elliptic.  The following proposition can be found in \cite[Appendix A]{Kholmatov:2024_ifb} and provides a characterization of elliptic $C^2$-anisotropies. 

\begin{proposition}\label{prop:elliptic_anis_propo}
For any $C^2$-anisotropy the following assertions are equivalent:
\begin{itemize}
\item[(a)] $\norm$ is elliptic;

\item[(b)] $\norm^o$ is $C^2$ and elliptic;

\item[(c)] there exists $\bar r\in(0,1)$ such that for any $z\in\p  \Wulff^\norm$ there exist $x_z,y_z\in\R^2$ such that 
\[
B_{\bar r}(x_z)\subset \Wulff^\norm \subset \cl{B_{1/\bar r}(y_z)}
\qquad\text{and}\qquad 
\p B_{\bar r}(x_z)\cap \p \Wulff^\norm =\p B_{1/\bar r}(y_z)\cap \p  \Wulff^\norm = \{z\}.
\]
\end{itemize} 
\end{proposition}

\noindent 
Another interesting class of anisotropies is introduced in \cite[Section 4]{BKhN:2017_cpaa}:

\begin{definition}
We say an anisotropy $\norm$ is \emph{partially monotone} if
$$
(x_1,y_1),(x_2,y_2)\in\R^2,\quad |x_1|\le |x_2|,\,\,\,|y_1|\le |y_2|\qquad \Longrightarrow\quad \norm(x_1,y_1)\le \norm(x_2,y_2).
$$
\end{definition}

According to \cite[Appendix A]{BKhN:2017_cpaa} the following statements are equivalent: 
\begin{itemize}
\item $\norm$ is partially monotone; 
\item $\norm^o$ is partially monotone;
\item $\norm(x_1,x_2)=\norm(|x_1|,|x_2|)$ for all $x_1,x_2\in\R.$ 
\end{itemize}
Thus, $\norm$ is partially monotone if and only if it is even in each coordinates separately.  Equivalently, $\norm$ is partially monotone if and only if its Wulff shape $\Wulff^\norm$ is symmetric with respect to the coordinates axes.

\subsection{Anisotropic total variation and perimeter}

Let $\norm$ be an anisotropy in $\R^2$ and $I\subseteq\R$ be an open interval. Recall that a function $u:I\to\R$ has locally bounded variation in $I,$ and we write $u\in BV_\loc(I),$ if its distributional derivative $Du$ is a Radon measure in $I.$ If, additionally, 
$u\in L^1(I)$  and $Du$ is a bounded Radon measure in 
$I$, then $u$ is called a function of bounded variation and is denoted by $u\in BV(I).$ 

Given $u\in BV_\loc(I),$ the anisotropic area of the graph of $u\in BV_\loc(I)$ is  defined by the $\norm$-total variation of the Radon measure $(-Du,\sL^1)$ in an open set $J\strictlyincluded I$ as  
\begin{equation*}
\cA_\norm(u,J):=\int_J \norm^o(-Du,\sL^1):= \sup\Big\{\int_{J} (uh_1' + h_2)ds:\,\, (h_1,h_2)\in C_c^1(J;\R^2),\,\,\|\norm(h_1,h_2)\|_\infty\le 1\}.
\end{equation*}
When $u\in \Wulff_\loc^{1,1}(I),$ the Radon-Nikodym theorem implies
$$
\cA_\norm(u,J) = \int_J \norm^o(-u',1)ds,
$$
and hence, in the case of the Euclidean anisotropy,
$$
\cA_{|\cdot|}(u,J) = \int_J \sqrt{1+{u'}^2}\,ds.
$$
Note that  \cite[p. 390]{Dalmaso:1980}
\begin{align}\label{relaxing}
\cA_\norm(u,I) = &\int_I \norm^o(-u',1)ds + \norm^o(D^su,0)(I) \nonumber \\
=  &\int_I \norm^o(-u',1)ds + \sum_{x\in J_u}  \norm^o(\be_1)|u^+(x)-u^-(x)| + \norm^o(D^cu,0)(I),  
\end{align}
where $D^su$ and $D^cu$ are the singular part of $Du$ 
with respect to $\sL^1$ and the Cantor part respectively, 
$J_u$ is the jump set of $u,$ $u^\pm(x)$ are the right and left traces of $u$ at $x$ and 
$$
\norm^o(\mu,0)(I) = \sup\Big\{\int_I 
\eta \,d\mu:\,\, \eta\in C_c(I),\,\,\|\norm(\eta,0)\|_\infty\le 1\Big\}
$$ 
is the partial $\norm$-total variation of a Radon measure $\mu$
in $I$.

A measurable set $E\subset\R^2$ is called of locally finite perimeter in an open set $\Omega\subseteq\R^2,$ and denoted as $E\in BV_\loc(\Omega;\{0,1\}),$ provided that the distributional derivative  $D\chi_E$ of its characteristic function $\chi_E$ is a Radon measure in $\Omega.$ If, additionally, $D\chi_E$ is a bounded Radon measure
in $\Omega$, then $E$ has finite perimeter. We denote by $\p^*E$ 
and $\nu_E$ the reduced boundary and the generalized outer unit normal 
of $E,$ respectively. If $\chi_E\in BV(\R^2),$ we write $E\in BV(\R^2;\{0,1\}).$ 
We refer for instance 
to \cite{AFP:2000, Giusti:1980, Maggi:2012, MM:1984_book} for more information on $BV$-functions and sets of finite perimeter.

We define the $\norm$-perimeter of $E$ in the open set $\Omega\subseteq\R^2$ as
$$
\cP_\norm(E,\Omega) = \int_{\Omega\cap \p^*E} \norm^o(\nu_E)d\sH^1.
$$
We also set $\cP_\norm(E):=\cP_\norm(E,\R^2).$

For a function  $u:I\to\R$ we write 
$$
\subgraph(u):=\{(s,t)\in I\times \R:\,\, \,u(s)>t\}\subset \R^2
$$
to denote the (strict) subgraph of $u$ (sometimes called hypograph). There is a natural connection between the anisotropic area of the graph and the anisotropic perimeter of the subgraph. 

\begin{lemma}[{\cite{Dalmaso:1980}}]\label{lem:area_equal_perimeter}
$u\in BV_\loc(I)$ if and only if $\subgraph(u)$ has locally finite perimeter in $I\times\R.$ Moreover, in either case, for any $J\strictlyincluded I,$
$$
\cA_\norm(u,J) =  \cP_\norm(\subgraph(u),J\times\R).
$$
\end{lemma}

\subsection{Local minimizers} 
In this section we recall the notion of $\Lambda$-local minimizer.

\begin{definition}
Let $\norm$ be an anisotropy on $\R^2$ and $\Lambda\ge0.$

\begin{itemize}
\item We call a function $u\in BV_\loc(I)$ a $\Lambda$-local minimizer of $\cA_\norm$ in $I$ if
$$
\cA_\norm(u,J) \le \cA_\norm(u+\psi,J) + \Lambda\int_J|\psi|ds
$$
whenever $J\in \cO_b(I)$ and $\psi\in C_c^1(J).$ 

\item For an open set $\Omega\subseteq\R^2$ and $\Lambda\ge0,$ we call a set $E\in BV_\loc(\Omega;\{0,1\})$ a $\Lambda$-local minimizer of $\cP_\norm$ in $\Omega$ if 
$$
\cP_\norm(E,\Omega') \le \cP_\norm(F,\Omega') + \Lambda|E\Delta F|
$$
for any open $\Omega'\strictlyincluded \Omega$ and $F\in BV_\loc(\Omega;\{0,1\})$ with $E\Delta F\strictlyincluded \Omega'.$
\end{itemize}
When $\Lambda=0,$ following the literature, we shortly call $u$ (resp. 
$E$) a local minimizer.
\end{definition} 

By approximation, one can show that $u\in BV_\loc(I)$ 
is a $\Lambda$-local minimizer of $\cA_\norm$ in $I$ if and only if
$$
\cA_\norm(u,J) \le \cA_\norm(v,J) + \Lambda\int_J|u-v|ds
$$
whenever $J\in \cO_b(I)$ and $v\in BV_\loc(I)$ with $\supp(u-v)\strictlyincluded J.$ 

\begin{remark}\label{rem:finite_variat_local}
If $I$ is bounded and $u\in L^\infty(I)\cap BV_\loc(I)$ is a $\Lambda$-local minimizer of $\cA_\norm$ in $I,$ then $u\in BV(I).$ Indeed, for any open interval $J\strictlyincluded I$ consider the test function $v=u\chi_{I\setminus J}.$ Then 
$$
\cA_\norm(u,J) \le \norm(\be_2)|J| + 4\|u\|_\infty + 2\Lambda\|u\|_\infty |J|.
$$
Now letting $J\nearrow I$ we find $\cA_\norm(u,I)<+\infty$ and hence $u\in BV(I).$ In particular, the traces of $u$ on $\p I$ are well-defined. 
\end{remark}

These two notions are linked as follows.

\begin{proposition}\label{prop:locmin_func_set}
Let $u\in BV_\loc(I).$ 
\begin{itemize}
\item If the subgraph $\subgraph(u)$ is a $\Lambda$-local minimizer of $\cP_\norm$ in $I\times\R,$ then $u$ is a $\Lambda$-local minimizer of $\cA_\norm$ in $I.$

\item If $\norm$ is partially monotone and $u$ is a $\Lambda$-local minimizer of $\cA_\norm$ in $I,$  then  $\subgraph(u)$ is a $\Lambda$-local minimizer of $\cP_\norm$ in $I\times\R.$  
\end{itemize}
\end{proposition}

At the moment  we do not have any explicit example showing the necessity of partial monotonicity in the second assertion of the  proposition.

\begin{proof}
Let  $\subgraph(u)$ be a $\Lambda$-local minimizer of $\cP_\norm$ in $I\times\R,$  and fix $J\in \cO_b(I)$ and $\psi\in C_c^1(J).$ Then 
$
\subgraph(u)\Delta \subgraph(u+\psi) \strictlyincluded J\times\R
$
and hence, for any bounded open set $\Omega'\strictlyincluded I\times\R$ compactly containing $\subgraph(u)\Delta \subgraph(u+\psi)$ we have 
$$
\cP_\norm(\subgraph(u),J\times\R) - \cP_\norm(\subgraph(u+\psi),J\times\R) = \cP_\norm(\subgraph(u),\Omega') - \cP_\norm(\subgraph(u+\psi),\Omega') \le \Lambda|\subgraph(u)\Delta \subgraph(u+\psi)|.
$$
By Lemma \ref{lem:area_equal_perimeter} and the equality 
$$
\int_J |u-v|ds = |[\subgraph(u)\Delta \subgraph(v)]\cap [J\times \R]|,
$$
the inequality above is equivalent to 
$$
\cA_\norm(u,J) -\cA_\norm(u+\psi,J) \le \Lambda \int_J|(u+\psi)-u|ds = \Lambda\int_J|\psi|ds,
$$
and hence  $u$ is a $\Lambda$-local minimizer of $\cA_\norm$ in $I.$

Conversely, assume that $\norm$ is partially monotone and $u$ is a $\Lambda$-local minimizer of $\cA_\norm$ in $ I.$
Let $F\in BV_\loc(I\times\R;\{0,1\})$ be such that $\subgraph(u)\Delta F\strictlyincluded J\times(a,b)\strictlyincluded I\times\R$ for some $J\in \cO_b(I)$ and $a,b\in\R.$ Let $v$ be the function, whose subgraph is the   vertical rearrangement of $F,$ i.e., 
$$
v(s) = a + \sH^1(\{x_2\in(a,b):\,\, (s,x_2)\in F\}),\quad s\in I.
$$
Note that by the definition of the rearrangement, for a.e. $s\in I,$ $v(s)$ satisfies 
\begin{equation*}
|u(s)-v(s)| = \sH^1((\subgraph(u)\Delta F)\cap \{x_1=s\})
\end{equation*}
so that by the Fubini-Tonelli theorem,
\begin{equation}\label{hstdzv}
\int_{J'}|u-v|ds = \int_{J'} \sH^1\big(
(\subgraph(u)\Delta F)\cap \{x_1=s\}\big)ds = |(\subgraph(u)\Delta F) \cap (J'\times\R)|\quad\text{for any $J'\in\cO_b(I)$}.
\end{equation}
Repeating the same arguments of \cite[Section 4]{BKhN:2017_cpaa}  (see also \cite{Miranda:1964}) we can show that $v\in BV_\loc^1(I),$ $\supp(u-v)\strictlyincluded J,$ 
\begin{equation}\label{ahstdv}
\sL^1 (J') \le \int_{J'\times\R}|\scalarp{D\chi_F}{\be_2}|\quad\text{and}\quad 
\int_{J'}|Dv| \le \int_{J'\times\R}|\scalarp{D\chi_F}{\be_1}|\quad\text{for any $J'\in \cO_b(I),$}
\end{equation}
where the Radon measures $-\scalarp{D\chi_F}{\be_1}$ and $-\scalarp{D\chi_F}{\be_2}$ are the horizontal and vertical 
components of $D\chi_F,$ which coincide with $\scalarp{\nu_F}{\be_1}\,\sH^1\res\p^*F$ and $\scalarp{\nu_F}{\be_2}\,\sH^1\res\p^*F,$ respectively. 
Since $\norm^o$ is partially monotone, by \eqref{ahstdv} 
\begin{align}\label{asfhfbb}
\cA_\norm(v,J) 
= & \int_{J'} \norm^o(-Dv,1) 
=
\int_{J'} \norm^o(|Dv|,\sL^1) 
\le   
\int_{J'\times\R} \norm^o\big(|\scalarp{D\chi_F}{\be_1}|,|\scalarp{D\chi_F}{\be_2}|\big) 
\nonumber
\\
= & 
\int_{J'\times\R} \norm^o\big(\scalarp{D\chi_F}{\be_1},
\scalarp{D\chi_F}{\be_2}\big) 
= 
\int_{J'\times\R} \norm^o(D\chi_F) = \cP_\norm(F,J'\times\R)
\end{align}
for all $J'\in \cO_b(I).$

Now, by Lemma \ref{lem:area_equal_perimeter}  and the $\Lambda$-local minimality of $u$,
\begin{align}\label{iwontbelieve1}
\cP_\norm(\subgraph(u),J\times \R) - \cP_\norm(\subgraph(v),J\times \R) 
= \cA_\norm(u,J) - \cA_\norm(v,J) 
\le \Lambda \int_J |u-v|ds.
\end{align}
Applying \eqref{hstdzv} and \eqref{asfhfbb} with $J'=J$ and recalling that $\subgraph(u)\Delta F \strictlyincluded J\times (a,b),$ from \eqref{iwontbelieve1} we conclude 
\begin{align*}
\cP_\norm(\subgraph(u),J\times (a,b)) - \cP_\norm(F,J\times (a,b)) = &\cP_\norm(\subgraph(u),J\times \R) - \cP_\norm(F,J\times \R) \\
\le  & \cP_\norm(\subgraph(u),J\times \R) - \cP_\norm(\subgraph(v),J\times \R) \le \Lambda|\subgraph(u)\Delta F|.
\end{align*}
Thus, by definition, $\subgraph(u)$ is a $\Lambda$-local minimizer of $\cP_\norm$ in $I\times\R.$
\end{proof}

Note that if $\norm$ is partially monotone, then $\cA_\norm(u,\cdot)=\cA_\norm(-u,\cdot),$ and hence  $u$ is $\Lambda$-local minimizer if and only if so is $-u.$ Thus, from Proposition \ref{prop:locmin_func_set} we get the following corollary.

\begin{corollary}\label{cor:inter_corol}
Let $u\in BV_\loc(I).$ For any partially monotone anisotropy $\norm,$  the following assertions are equivalent:
\begin{itemize}
\item $u$ is a $\Lambda$-local minimizer of $\cA_\norm$ in $I;$

\item  $-u$ is a $\Lambda$-local minimizer of $\cA_\norm$ in $I;$ 

\item the subgraph $\subgraph(u)$ of $u$ is a $\Lambda$-local minimizer of $\cP_\norm$ in $I\times\R;$

\item  the (strict) epigraph $\epigraph(u):=\{(s,t)\in I\times\R:\,\, u(s)<t\}$ is a $\Lambda$-local minimizer of $\cP_\norm$ in $I\times\R.$
\end{itemize}
\end{corollary}

\subsection{Density estimates}

The proof of the next lemma is well-known in the literature (see e.g. \cite{DeM:2015,Kholmatov:2024_ifb}) and can be proven, for instance, using the filling-in or cutting-out with balls.

\begin{lemma}\label{lem:density_estimates}
Given an anisotropy $\norm,$ $\Lambda\ge0$ and an open set $\Omega\subseteq\R^2,$ let $E\in BV_\loc(\Omega;\{0,1\})$ be a $\Lambda$-local minimizer of $\cP_\norm$ in $\Omega.$ Assume that $E=E^{(1)},$ i.e., $E$ coincides with its Lebesgue points. Then for any $\Omega'\strictlyincluded\Omega$ there exist constants $r_0:=r_0(\norm,\Lambda,\dist(\p \Omega',\p\Omega))>0$ and $q_0:=q_0(\norm,\Lambda)\in(0,1/2)$ such that 
\begin{equation*}
P(E,B_r(x)) \le \frac{r}{q_0},\quad x\in\Omega',\,\,\,r\in(0,r_0), 
\end{equation*}
\begin{equation*}
q_0 \le   \frac{|E\cap B_r(x)|}{|B_r(x)|} \le 1-q_0,\quad x\in \p E,\,\,\, r\in(0,r_0),
\end{equation*}
and 
\begin{equation*}
P(E,B_r(x)) \ge q_0 r,\quad x\in \p E,\,\,\, r\in(0,r_0).
\end{equation*}
\end{lemma}

From Lemma \ref{lem:density_estimates} and a
covering argument we immediate deduce that every $\Lambda$-local minimizer $E=E^{(1)}$ in $\Omega$ satisfies 
\begin{equation}\label{topo_boundo}
\p E=\cl{\p^*E},\quad \sH^1\big(\Omega'\cap (\p E\setminus \p^*E)
\big) = 0 \quad\text{and}\quad \sH^1\big(\Omega'\cap 
(\cl{E}\setminus \mathring{E})\big) <+\infty\quad\text{for any open $\Omega'\strictlyincluded \Omega.$}
\end{equation}
In particular, possibly changing a negligible set, $E$ can be assumed open or closed. 

\begin{remark}\label{rem:lip_regular_sets}
Any $\Lambda$-local minimizer $E$ of $\cP_\norm$ in $\Omega$ satisfies
$$
\cP_\norm(E,B_\rho(x)) \le \cP_\norm(F,B_\rho(x)) + \Lambda \sqrt{\pi} \rho\sqrt{|E\Delta F|}
$$
whenever $x\in \p E,$ $B_\rho(x)\strictlyincluded \Omega$ and $E\Delta F\strictlyincluded B_\rho(x).$ Thus, $E$ is $\omega$-minimal in the sense of  \cite{NP:2005_ampa} with $\omega(\rho)=\Lambda\sqrt{\pi}\rho.$ In particular, by  \cite[Theorem 3.4]{NP:2005_ampa}, the set $\Sigma$ of all points $x\in \Omega\cap \p E$ around which $\Omega\cap \p E$ is not a Lipschitz graph is discrete and is empty if $\Wulff^\norm$ is not a quadrilateral\footnote{In fact, 
$\omega$-minimal sets are defined for any anisotropy, not necessarily even and \cite[Theorem 3.4]{NP:2005_ampa} shows that in general $\Sigma$ is discrete. Moreover, if $W^\norm$ is neither a triangle nor a quadrilateral, then $\Sigma$ is empty.}. 
\end{remark}

\section{Classification of 
local minimizers}\label{sec:class_local_mini}

In this section we classify the minimizers of $\cA_\norm,$ i.e., study functions $u\in BV_\loc(I)$ satisfying 
$
\cA_\norm(u,J)\le \cA_\norm(v,J)
$
for any open set $J\Subset I$ and $v\in BV_\loc(I)$ with $\supp(u-v)\Subset J.$ 

\begin{theorem}[Characterization of local minimizers]\label{teo:0_loc_minos}
Let $\norm$ be a partially monotone anisotropy, $I\subseteq\R$ be an interval and $u\in BV_\loc(I).$ Let $\Gamma_u:=(I\times\R)\cap\cl{\p^*\subgraph(u)}$ be the generalized graph of $u$ and $\nu_{\subgraph(u)}:\Gamma_u\to\S^1$ be the unit normal field, outer to $(I\times\R)\cap \subgraph(u),$ defined $\sH^1$-a.e. on $\Gamma_u.$ 
Then $u$ is a local minimizer 
of $\cA_\norm$ in $I$ if and only if there exists a vector $N\in\R^2$  such that 
\begin{equation}\label{vectorN}
\norm(N)=1\qquad\text{and}\qquad 
\scalarp{N}{\nu_{\subgraph(u)}} = \norm^o(\nu_{\subgraph(u)})\quad\text{$\sH^1$-a.e. on $\Gamma_u.$}
\end{equation}
Moreover, $u$ is monotone in $I.$
\end{theorem}

In the literature the vector $N$ satisfying \eqref{vectorN} is sometimes called a Cahn-Hoffman vector field associated to the rectifiable  curve  $\Gamma_u.$

\begin{proof}
We expect this result to be well-known in the literature; for  completeness we provide the proof.
\smallskip

$\Rightarrow.$ We apply a calibration argument as in \cite[Example 2.4]{BKhN:2017_cpaa}. Assume that there exists a vector  $N$   satisfying \eqref{vectorN} and let $F\in BV_\loc(I\times\R;\{0,1\})$ be such that $F\Delta \subgraph(u)\strictlyincluded J\times\R$ for some $J\strictlyincluded I.$ Then 
\begin{align*}
\cP_\norm(F,J\times\R) &= \int_{(J\times\R)\cap \p^*F} \norm^o(\nu_F)d\sH^1 \ge \int_{(J\times\R)\cap \p^*F} \scalarp{\nu_F}{N}\,d\sH^1. 
\end{align*}
On the other hand, by the divergence theorem
$$
0= \int_{F\setminus \subgraph(u)} \div N\,dx- \int_{\subgraph(u)\setminus F} \div N\,dx= \int_{(J\times\R)\cap \p^*F} \scalarp{\nu_F}{N}\,d\sH^1 -
\int_{(J\times\R)\cap \p^*\subgraph(u)} \scalarp{\nu_{\subgraph(u)}}{N}\,d\sH^1,
$$
and thus 
\begin{align*}
\int_{(J\times\R)\cap \p^*F} \scalarp{\nu_F}{N}\,d\sH^1 =\int_{(J\times\R)\cap \p^*\subgraph(u)} \scalarp{\nu_{\subgraph(u)}}{N}\,d\sH^1 = 
\int_{(J\times\R)\cap \p^*\subgraph(u)} \norm^o(\nu_{\subgraph(u)})\,d\sH^1 =\cP_\norm(\subgraph(u),J\times\R).
\end{align*}
Hence $\subgraph(u)$ is a local minimizer of $\cP_\norm$ in $I\times\R.$ Then Corollary \ref{cor:inter_corol} implies that $u$ is a local minimizer of $\cA_\norm$ in $I.$
\smallskip

$\Leftarrow.$ Assume that $u$ is a local minimizer of $\cA_\norm$ in $I.$ Since $|Du|(J)<+\infty$ for any open interval $J\strictlyincluded I,$ we have $u\in L^\infty(J).$ In particular, $u\in BV_\loc(I).$ 

Let $J\strictlyincluded I$ be any interval, whose boundary points are not on the jump set of $u;$ such an interval exists because $u$ has at most countably many jumps. Let $v$ be the function such that $v=u$ in $I\setminus J$ and linear in $J$ such that the traces of $u$ and $v$ on $\p J$ coincide. By the local boundedness of $u,$ $\subgraph(u)\Delta\subgraph(v)\strictlyincluded I\times\R.$ Then by the local minimality of $u$ and the anisotropic minimality of segments \cite{FFLM:2011_jmpa}, 
$$
\cP_\norm(\subgraph(u),J\times\R) \ge \cP_\norm(\subgraph(v),J\times\R) \ge \cP_\norm(\subgraph(u),J\times\R),
$$
and hence
$\cP_\norm(\subgraph(u),J\times\R) =  \cP_\norm(\subgraph(v),J\times\R).$ 
Choose a $N\in\R^2$ satisfying $\norm(N)=1$ and $\scalarp{\nu_{[p,q]}}{N}=\norm^o(\nu_{[p,q]})$ on $[p,q].$ As above, by the divergence formula 
$$
0= \int_{\subgraph(v)\setminus \subgraph(u)} \div N\,dx- \int_{\subgraph(u)\setminus \subgraph(v)} \div N\,dx= \int_{(J\times\R)\cap \p^*\subgraph(v)} \scalarp{\nu_F}{N}\,d\sH^1 -
\int_{(J\times\R)\cap \p^*\subgraph(u)} \scalarp{\nu_{\subgraph(u)}}{N}\,d\sH^1.
$$
Thus, 
\begin{multline}\label{vtbn13d}
\cP_\norm(\subgraph(v),J\times\R) =
\cP_\norm(\subgraph(u),J\times\R) = \int_{(J\times\R)\cap \p^*\subgraph(u)} \norm^o(\nu_{\subgraph(u)})d\sH^1 \ge \int_{(J\times\R)\cap \p^*\subgraph(u)} \scalarp{\nu_F}{N}\,d\sH^1 \\
=\int_{(J\times\R)\cap \p^*\subgraph(v)} \scalarp{\nu_{\subgraph(v)}}{N}\,d\sH^1 = 
\int_{(J\times\R)\cap \p^*\subgraph(v)} \norm^o(\nu_{\subgraph(v)})\,d\sH^1 =\cP_\norm(\subgraph(v),J\times\R),
\end{multline}
where in the fourth equality we used that $v$ is linear in $J.$ 
Thus, all inequalities in \eqref{vtbn13d} are in fact equalities. Since $\norm^o(\nu_{\subgraph(u)}) \ge \scalarp{\nu_{\subgraph(u)}}{N}$ $\sH^1$-a.e. on $(J\times\R)\cap \p^*\subgraph(u)$ and 
$$
\int_{(J\times\R)\cap \p^*\subgraph(u)} \norm^o(\nu_{\subgraph(u)})d\sH^1 = \int_{(J\times\R)\cap \p^*\subgraph(u)} \scalarp{\nu_F}{N}\,d\sH^1,
$$
from the Chebyshev inequality it follows that $\norm^o(\nu_{\subgraph(u)}) = \scalarp{\nu_F}{N}$ $\sH^1$-a.e. on $(J\times\R)\cap \p^*\subgraph(u).$  Now,
 consider a sequence  $J_k\nearrow I$ of open relatively
compact intervals and the  associated constant vectors $N_k\in \p \Wulff^\norm.$ Notice that each $N_k$ satisfies 
\begin{equation}\label{ahsgstdv}
\norm(N_k)=1\qquad\text{and}\qquad 
\norm^o(\nu_{\subgraph(u)}) = \scalarp{\nu_F}{N_k}\quad\text{$\sH^1$-a.e. on $(J_k\times\R)\cap \p^*\subgraph(u).$}
\end{equation}
Since $\p\Wulff^\norm$ is compact, there is no loss of generality in assuming $N_k\to N$ for some $N\in\p\Wulff^\norm.$ Note that, 
given $\bar k \in \mathbb N$, all $N_k$ with $k\ge\bar k$ satisfy  \eqref{ahsgstdv} in $J_{\bar k}.$ Since $N_k$ appear linearly in the second relation of \eqref{ahsgstdv}, it follows that any vector in the  closed convex hull $K_{\bar k}$ of $\cup_{k\ge\bar k}N_k$ also satisfies \eqref{ahsgstdv}. Clearly, $N$ belongs to $K_{\bar k}$ for all $\bar k.$ As $J_k\nearrow I,$ it follows that $N$ satisfies \eqref{vectorN}.

Finally, let us show that $u$ is monotone, i.e., it admits a monotone representative. Indeed, suppose that there exist
 $(a,b)\strictlyincluded I$ and $t\in\R$ such that $(a,t),(b,t)\in \Gamma_u.$ Let us define the competitor $v = u\chi_{I\setminus (a,b)} + t\chi_{(a,b)}.$ By the local minimality of $u,$ for any open interval $J$ with $(a,b)\strictlyincluded J\strictlyincluded I$ we have 
\begin{multline}\label{ctbunm}
0\le \cA_\norm(v,J) - \cA_\norm(u,J) = \int_{(a,b)} \Big(\norm^o(0,1)d\sL^1 - \norm^o(-Du,1)\Big) \\
+ \norm^o(\be_1)\Big( |v^+(a)-v^-(a)| - |u^+(a)-u^-(a)| + |v^+(b)-v^-(b)| - |u^+(b)-u^-(b)|\Big),
\end{multline}
where in the equality we used \eqref{relaxing}.
By the definition of $v$ and the choice of $t,$ $u^-(a)=v^-(a),$ $u^+(b)=v^+(b),$ $v^+(a)=v^-(b)=t$ and 
$$
|v^+(a)-v^-(a)| \le  |u^+(a)-u^-(a)|,\qquad|v^+(b)-v^-(b)| \le |u^+(b)-u^-(b)|.
$$
Moreover,  by the partial monotonicity of $\norm^o$ we have 
$$
\int_a^b \norm^o(-u',1)ds \ge \int_a^b \norm^o(0,1)ds,
$$
and hence, by \eqref{ctbunm} and \eqref{relaxing} we have 
\begin{multline*}
0\le \norm^o(\be_1)\Big(|u^+(a)-u^-(a)| - |v^+(a)-v^-(a)| + |u^+(b)-u^-(b)| -|v^+(b)-v^-(b)|\Big) \\
\le  \int_a^b \norm^o(0,1)ds - \int_a^b \norm^o(-u',1)ds - \norm^o(D^su,0)(a,b) \le 0. 
\end{multline*}
Thus, all inequalities are in fact equalities, $u^+(a)=u^-(b)=t,$ $u'=0$ a.e. in $(a,b)$ and $D^su=0.$ This implies $u=v$ in $(a,b).$ This observation shows that for any $\lambda\in \R$ the set $\{u=\lambda\}$ is either empty, or one point or an interval. Therefore, $u$ is monotone.
\end{proof}

\begin{example}[Strictly convex anisotropies]
Assume that $\norm^o$ is strictly convex, i.e., 
$$
\norm^o(x+y)<\norm^o(x)+\norm^o(y)\quad\text{whenever $|x|=|y|$ with $x\ne \pm y.$}
$$ 
Then for any interval $I\subseteq\R,$ the function $u\in BV_\loc(I)$ is a local minimizer of $\cA_\norm$ if and only if $u$ is linear. Indeed, by the strict convexity of $\norm,$ for any $N\in \p \Wulff^\norm$ there exists a unique $\nu\in \S^1$ such that $\scalarp{N}{\nu} = \norm^o(\nu).$ Thus, by Theorem \ref{teo:0_loc_minos} $u$ is a local minimizer of $\cA_\norm$ in $I$ if any only if $\Gamma_u$ admits a constant unit normal $\sH^1$-a.e., which is equivalent to say that $u$ is linear.
\end{example}

\begin{example}[Square anisotropy]
Let $\Wulff^\norm=[-1,1]^2$ and $I\subseteq\R$ be an interval, Then $u$ is  a local minimizer of $\cA_\norm$ if and only if $u$ is monotone. Indeed, by Theorem \ref{teo:0_loc_minos} every local minimizer is monotone. Conversely, consider any nondecreasing function $u:I\to\R.$ By monotonicity, the unit normals $\nu_u$ to $\Gamma_u$ lie in the smaller closed arc of $\S^1$ between $-\be_1$ (jump part) and $\be_2$ (constant part). 
Thus, any constant vector $N=(-1,1)\in\p\Wulff^\norm$ satisfies 
$$
\scalarp{N}{\nu_u} = |\scalarp{\nu_u}{\be_1}|+ |\scalarp{\nu_u}{\be_2}| =  \norm^o(\nu_u)\quad \text{$\sH^1$-a.e. on $\Gamma_u.$}
$$
Hence, by Theorem \ref{teo:0_loc_minos}, $u$ is a local minimizer.
\end{example}

\begin{example}[Lens-shaped anisotropies]
Given $a>0,$ let $\gamma\in C^1([-a,0])$ be a strictly increasing concave function with $\gamma(-a)=0$ and $\gamma'(0)=0.$ Let $\Wulff^\norm$ be the convex set symmetric with respect to the 
coordinate axes such that $((-\infty,0)\times(0,+\infty))\cap \p\Wulff^\norm$ is the graph of $\gamma.$ Let $I\subseteq \R$ be an interval. Then $u\in BV_\loc(I)$ is a local minimizer of $\cA_\norm$ in $I$ if any only if either $u$ is linear, or 
$u$ is monotone and piecewise linear, and all segments/half-lines of its graph are tangent\footnote{I.e., their normal belongs to $\p\norm(\be_1).$} to $\Wulff^\norm$ at exactly one of the two points $\pm\frac{\be_1}{\norm(\be_1)}.$
\end{example}

\section{Regularity of $\Lambda$-minimizers}\label{sec:regul_lambda_min}

Now consider the case $\Lambda>0$. In this case a general 
characterization of $\Lambda$-local minimizers as in 
Theorem \ref{teo:0_loc_minos} seems not available. 
In this section, under some assumptions of $\norm,$ we show that if the $L^\infty$-norm of a $\Lambda$-minimizer of $\cA_\norm$ in $I$ is sufficiently small, then $u$ is Lipschitz in $I.$

\begin{theorem}[Regularity of $\Lambda$-minimizers]\label{teo:jums_absent}
Let $\norm$ be  a partially monotone anisotropy such that $\Wulff^\norm$ does not have vertical facets (so that $\pm\be_1$ is an ``outer normal'' to $\Wulff^\norm$ only at $\frac{\pm\be_1}{\norm(\be_1)}$). Given a bounded open interval $I\subset\R$ and $\Lambda>0,$ let $u\in BV_\loc(I)$ be a $\Lambda$-local minimizer of $\cA_\norm$ in $I$ satisfying 
\begin{equation}\label{asdfvv}
\|u\|_\infty < \min\Big\{\tfrac{\alpha_0\norm(\be_1)}{4\Lambda}, \tfrac{\alpha_0}{2\Lambda\norm(\be_2)}, \, \tfrac{|I|\norm(\be_1)}{4\Lambda \norm(\be_2)}\Big\},
\end{equation}
with 
\begin{equation}\label{def:alfa0}
\alpha_0=\alpha_0(\norm):=\tfrac{\cP_\norm(\Wulff^\norm)}{(2\cP_\norm(\Wulff^\norm)+ 1) \,\sqrt{|\Wulff^\norm|} }>0.
\end{equation}
Then $u$ is Lipschitz in $I.$ Moreover, if $\norm$ is $C^2$ and elliptic, then $u\in C^{1,1}(I),$ that is, $u$ is continuously differentiable and its derivative $u'$ is Lipschitz in $I.$
\end{theorem}

Note that every partially monotone elliptic anisotropy satisfies the assumption of the theorem. Moreover, by Remark \ref{rem:finite_variat_local}, $u\in BV(I).$ Furthermore, by the partial monotonicity of $\norm$ and Corollary \ref{cor:inter_corol}, the  subgraph and epigraph of $u$ are $\Lambda$-local minimizers of $\cP_\norm$ in $I\times\R.$ 
Notice also 
that when $\Wulff^\norm$ is not a quadrilateral,  Remark \ref{rem:lip_regular_sets} ensures that the boundaries of these sets in $I\times\R$ are locally given by a Lipschitz graph. However, this graphicality property of $\subgraph(u)$ and $\epigraph(u)$ does not yield that $u$ itself is Lipschitz; for instance, $u$ may have jump discontinuities (see the function $u_{a,b}$ in \eqref{ahst6bnega}).
 
The proof of Theorem  \ref{teo:jums_absent} is postponed to the end of the section after some ancillary results. We start with the following property of Wulff shapes.

\begin{lemma}\label{lem:rel_iso_inq}
Let a bounded $D\in BV(\R^2;\{0,1\})$ and a Wulff shape $\Wulff_r^\norm$ with $r>0$ be such that $\Wulff_r^\norm \cap D=\emptyset$ and $\sH^1(\p^* D\cap \p \Wulff_r^\norm)>0.$  Then 
\begin{equation}\label{atertas}
\int_{\p^*D\setminus \p \Wulff_r^\norm}\norm^o(\nu_D)\,d\sH^1 - \int_{\p^*D\cap \p \Wulff_r^\norm} \norm^o(\nu_D)\,d\sH^1 \ge \alpha_0 \, \min\Big\{\sqrt{|D|},\frac{|D|}{r}\Big\}, 
\end{equation}
where $\alpha_0>0$ is given in \eqref{def:alfa0}.
\end{lemma}

\begin{proof}
The proof is similar to \cite[Proposition 6.3]{FFLM:2011_jmpa}.  Recall that the isoperimetric inequality (see e.g. \cite{Fonseca:1991}) says 
\begin{equation}\label{isop_ineq_aniso}
\cP_\norm(E) \ge c_\norm\,\sqrt{|E|},\quad E\in BV(\R^2;\{0,1\}),
\end{equation}
where 
$
c_\norm:=\frac{\cP_\norm(\Wulff^\norm)}{\sqrt{|\Wulff^\norm|}},
$
and the equality in \eqref{isop_ineq_aniso} holds if and only if $E = x+ r\Wulff^\norm=W_r^\norm(x)$ for some $x\in\R^2$ and $r\ge0.$
When $\norm$ is Euclidean, $c_\norm=\sqrt{4\pi}.$

Fix any $\alpha>1.$ First consider the case
\begin{equation}\label{temp1}
\int_{\p^*D\setminus \p \Wulff_r^\norm}\norm^o(\nu_D)\,d\sH^1 \le \alpha \int_{\p^*D\cap \p \Wulff_r^\norm} \norm^o(\nu_D)\,d\sH^1 
\end{equation}
Since $\Wulff_r^\norm \cap D=\emptyset,$ by \cite[Theorem 16.3]{Maggi:2012} one has $\p^*(D\cup \Wulff_r^\norm) \approx_{\sH^1} [\p^*D\setminus \Wulff_r^\norm]\cup [\p^*\Wulff_r^\norm\setminus \p^*D],$ where $A\approx_\mu B$ stands for $\mu(A\Delta B)=0.$ Therefore, 
\begin{multline}\label{tempsodp1}
\int_{\p^*D\setminus \p \Wulff_r^\norm}\norm^o(\nu_D)\,d\sH^1 - \int_{\p^*D\cap \p \Wulff_r^\norm} \norm^o(\nu_D)\,d\sH^1 = 
\int_{\p^*(D\cup \Wulff_r^\norm)}\norm^o(\nu_{D\cup \Wulff_r^\norm})\,d\sH^1 - \int_{\p \Wulff_r^\norm} \norm^o(\nu_{\Wulff_r^\norm})\,d\sH^1 \\
= \cP_\norm(D\cup \Wulff_r^\norm) - \cP_\norm(\Wulff_r^\norm) \ge c_\norm\Big(\sqrt{|D\cup \Wulff_r^\norm|} - \sqrt{|\Wulff_r^\norm|}\Big) = 
\tfrac{c_\norm|D|}{\sqrt{|D\cup \Wulff_r^\norm|} + \sqrt{|\Wulff_r^\norm|}},
\end{multline}
where in the first inequality we used the isoperimetric inequality \eqref{isop_ineq_aniso}.  
Moreover, by \eqref{temp1} and again by the isoperimetric inequality and the assumption $D\cap \Wulff_r^\norm=\emptyset,$ 
\begin{multline*}
c_\norm \sqrt{|D\cup \Wulff_r^\norm|} \le P_\norm(D\cup \Wulff_r^\norm) = 
\int_{\p^*D\setminus \p \Wulff_r^\norm}\norm^o(\nu_D)\,d\sH^1 + \int_{\p^*\Wulff_r^\norm\setminus \p D} \norm^o(\nu_{\Wulff_r^\norm})\,d\sH^1 \\
\le \alpha \int_{\p^*D\cap \p \Wulff_r^\norm} \norm^o(\nu_D)\,d\sH^1 + \int_{\p^*\Wulff_r^\norm\setminus \p D} \norm^o(\nu_{\Wulff_r^\norm})\,d\sH^1  \le \alpha  \int_{\p^*\Wulff_r^\norm} \norm^o(\nu_{\Wulff_r^\norm})d\sH^1 = \alpha \cP_\norm(\Wulff_r^\norm). 
\end{multline*}
Thus, recalling
$
c_\norm\sqrt{|\Wulff_r^\norm|} = \cP_\norm(\Wulff_r^\norm),
$ from \eqref{tempsodp1} we get  
\begin{equation}\label{temp_sol1}
\int_{\p^*D\setminus \p \Wulff_r^\norm}\norm^o(\nu_D)\,d\sH^1 - \int_{\p^*D\cap \p \Wulff_r^\norm} \norm^o(\nu_D)\,d\sH^1  \ge 
\tfrac{c_\norm|D|}{(\alpha +1)\cP_\norm(\Wulff_r^\norm)} = 
\tfrac{|D|}{(\alpha +1)r\sqrt{|\Wulff^\norm|} } .
\end{equation}

\medskip
Now consider the case
\begin{equation}\label{temp2}
\int_{\p^*D\setminus \p \Wulff_r^\norm}\norm^o(\nu_D)\,d\sH^1 > \alpha  \int_{\p^*D\cap \p \Wulff_r^\norm} \norm^o(\nu_D)\,d\sH^1 
\end{equation}
so that 
\begin{equation}\label{kadikal}
\int_{\p^*D\setminus \p \Wulff_r^\norm}\norm^o(\nu_D)\,d\sH^1 - \int_{\p^*D\cap \p \Wulff_r^\norm} \norm^o(\nu_D)\,d\sH^1 > \Big(1-\frac{1}{\alpha }\Big) \int_{\p^*D\setminus \p \Wulff_r^\norm}\norm^o(\nu_D)\,d\sH^1.
\end{equation}
Then by the isoperimetric inequality, \eqref{kadikal} and \eqref{temp2}   we get 
\begin{align*}
c_\norm\,|D|^{1/2} \le & \cP_\norm(D) = \int_{\p^* D\setminus \p \Wulff_r^\norm}\norm^o(\nu_D)\,d\sH^1  + \int_{\p^*D\cap \p \Wulff_r^\norm}\norm^o(\nu_D)\,d\sH^1 \\
\le & \Big(1 + \frac{1}{\alpha }\Big)\,\int_{\p^* D\setminus \p \Wulff_r^\norm}\norm^o(\nu_D)\,d\sH^1  \le 
\frac{\alpha +1}{\alpha -1}\,\Big( \int_{\p^*D\setminus \p \Wulff_r^\norm}\norm^o(\nu_D)\,d\sH^1 - \int_{\p^*D\cap \p \Wulff_r^\norm} \norm^o(\nu_D)\,d\sH^1 \Big).
\end{align*}
Combining this inequality with \eqref{temp_sol1} we deduce
$$
\int_{\p^*D\setminus \p \Wulff_r^\norm}\norm^o(\nu_D)\,d\sH^1 - \int_{\p^*D\cap \p \Wulff_r^\norm} \norm^o(\nu_D)\,d\sH^1  
\ge 
\min\Big\{ \tfrac{(\alpha -1)c_\norm \sqrt{|D|}}{\alpha +1}, 
\tfrac{|D|}{(\alpha +1)r\sqrt{|\Wulff^\norm|}}\Big\}.
$$
Now, choosing $\alpha :=1+\frac{1}{\cP_\norm(\Wulff^\norm)}$ we get \eqref{atertas}.
\end{proof}

Next we ``improve'' the regularity of subgraph and epigraph of $\Lambda$-minimizers $u.$ For simplicity,  let $E_u$ and $F_u$ be the open representatives of $\subgraph(u)$ and $\epigraph(u),$ respectively (see Lemma \ref{lem:density_estimates}), and let 
$$
\Gamma_u:=(I\times \R)\cap \p E_u=(I\times\R)\cap \p F_u
$$
be the (generalized) graph of $u$ in $I.$ By Remark \ref{rem:lip_regular_sets} $\Gamma_u$ is a locally Lipschitz curve\footnote{Possibly out of a discrete set when $\Wulff^\norm$ is a quadrilateral.} (thus an arcwise connected set) and, as the traces of $u$ on $\p I$ are well-defined (see Remark \ref{rem:finite_variat_local}), its topological closure $\cl \Gamma_u$ consists of the union of $\Gamma_u$ and two points on $\p I,$ whose vertical coordinates correspond to the traces  of $u.$

\begin{proposition}[Contact $\norm$-ball condition]\label{prop:uniform_wulff_locmin}
Let $\norm$ be a partially 
monotone anisotropy in $\mathbb R^2$, $I\subset \R$ be a bounded open interval, $\Lambda>0$ and $u\in BV_\loc(I)$ be a $\Lambda$-local minimizer of $\cA_\norm$ in $I$ with 
\begin{equation}\label{u_L_chek_baho}
\|u\|_\infty \le  \frac{\alpha_0\norm(\be_1)}{4\Lambda},
\end{equation}
where $\alpha_0$ is given by \eqref{def:alfa0}. Then for any $r\in(0,\frac{\alpha_0}{\Lambda}):$ 
\begin{itemize}
\item[(a)] if $\oWulff_r^\norm(y)\cap F_u=\emptyset$ with $\Gamma_u\cap \p \oWulff_r^\norm(y)\ne\emptyset,$ then $\cl \Gamma_u \cap \p \Wulff_r^\norm(y)$ is connected (possibly singletons);

\item[(b)] if $\oWulff_r^\norm(z)\cap E_u=\emptyset$ with $\Gamma_u\cap \p \oWulff_r^\norm(z)\ne\emptyset,$ then $\cl \Gamma_u \cap \p \Wulff_r^\norm(z)$ is connected (possibly a singleton);

\item[(c)] for any $x\in\Gamma_u$ there exist $\norm$-balls $\Wulff_r^\norm(y)$ and $\Wulff_r^\norm(z)$ such that  $\oWulff_r^\norm(y)\cap E_u=\emptyset,$ $\oWulff_r^\norm(z)\cap F_u=\emptyset,$ and $\Gamma_u\cap \p\Wulff_r^\norm(y)$ and $\Gamma_u \cap \p \Wulff_r^\norm(z)$ are connected sets (possibly a singleton) containing $x.$

\end{itemize}

\end{proposition}
 
\begin{wrapfigure}[13]{l}{0.3\textwidth}
\includegraphics[width=0.27\textwidth]{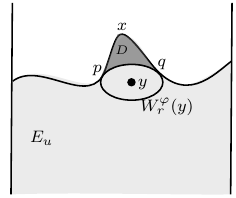}
\caption{\small }\label{fig:wulff_touch}
\end{wrapfigure}
Note that the $\norm$-balls $\oWulff_r^\norm(y)$ and $\oWulff_r^\norm(z)$ need not to lie in $I\times\R.$ The assertion (c) is related to the $r\Wulff^\norm$-condition in the literature, see e.g. \cite[Definition 5]{BCChN:2006}.
 
\begin{proof}
(a) Assume by contradiction that there exist $r\in(0,\frac{\alpha_0}{\Lambda})$ and a $\norm$-ball $\Wulff_r^\norm(y)$ such that $\oWulff_r^\norm(y)\cap F_u=\emptyset$ 
and the intersection $\cl \Gamma_u \cap \p\oWulff_r^\norm(y)$ is not connected. Let us denote by $p\ne q$ the upmost left and upmost right points of this intersection; 
in case there are several upmost left and/or upmost right points we select
 those with the smallest vertical coordinate. Note that $p,q\in \cl{I}\times\R.$

Let $D$ be the nonempty open set enclosed by the subcurves of $\cl \Gamma_u$ and $\p \oWulff_r^\norm(y)$ between $p$ and $q,$ not intersecting $\oWulff_r^\norm(y),$ see Fig. \ref{fig:wulff_touch}. Since $W^\norm$ is symmetric with respect to the coordinate axes, $y\pm r\frac{\be_1}{\norm(\be_1)}$ are the upmost left and upmost right points of $\Wulff_r^\norm(y)$ and hence,
$$
\scalarp{y-r\tfrac{\be_1}{\norm(\be_1)}}{\be_1} \le \min\{\scalarp{p}{\be_1},\scalarp{q}{\be_1}\} \le \max\{\scalarp{p}{\be_1},\scalarp{q}{\be_1}\} \le \scalarp{y-r\tfrac{\be_1}{\norm(\be_1)}}{\be_1}.
$$
Thus, recalling that $p$ and $q$ lie on the graph of $u,$
$$
D\subset \Big[\scalarp{y}{\be_1}-\tfrac{r}{\norm(\be_1)},\scalarp{y}{\be_1} + \tfrac{r}{\norm(\be_1)}\Big]\times [-\|u\|_\infty,\|u\|_\infty]
$$
and therefore 
\begin{equation}\label{atsz6fhff}
0<|D| \le \frac{4\|u\|_\infty r}{\norm(\be_1)}.
\end{equation}

First assume that 
$$
D\strictlyincluded I\times\R.
$$ 
Then by $\Lambda$-minimality, for any open set $\Omega'\strictlyincluded I\times\R,$ compactly contaning $D\subset E_u,$ we have
\begin{equation}\label{gstzfvb}
\cP_\norm(E_u,\Omega') \le \cP_\norm(E_u\setminus D,\Omega') + \Lambda|D|.
\end{equation}
Since $D\cap \Wulff_r^\norm (y) =\emptyset$ and $\sH^1(\p^*D\cap \p 
\Wulff_r^\norm(y)
)>0$ (because $\Wulff_r^\norm(y)$ 
touches $\p E_u$ at two different points), we can apply Lemma \ref{lem:rel_iso_inq} to get 
\begin{equation*}
\cP_\norm(E_u,\Omega') -\cP_\norm(E_u\setminus D,\Omega') = 
\int_{\p^*D\setminus \p \Wulff_r^\norm(y)}\norm^o(\nu_D)\,d\sH^1 - \int_{\p^*D\cap \p \Wulff_r^\norm(y)} \norm^o(\nu_D)\,d\sH^1 \ge \alpha_0 \, \min\Big\{\sqrt{|D|},\frac{|D|}{r}\Big\}, 
\end{equation*}
and thus 
\begin{equation}\label{ahczs6gg}
\alpha_0 \, \min\Big\{\sqrt{|D|},\frac{|D|}{r}\Big\} \le \Lambda|D|.
\end{equation} 
Now, if $\sqrt{|D|}\le \frac{|D|}{r},$ then by \eqref{ahczs6gg}, \eqref{atsz6fhff} and the  assumption $r<\frac{\alpha_0}{\Lambda}$ we have 
$$
\alpha_0 \le \Lambda\sqrt{|D|} < \Lambda\sqrt{ \tfrac{4\|u\|_\infty \alpha_0 }{\Lambda\norm(\be_1)} }\qquad\text{so that}\qquad \|u\|_\infty > \tfrac{\alpha_0\norm(\be_1)}{4\Lambda},
$$
which contradicts  \eqref{u_L_chek_baho}. On the other hand, if $\frac{|D|}{r}<\sqrt{|D|},$ then again by \eqref{ahczs6gg} and the choice of $r,$
$$
\frac{\alpha_0}{\Lambda} \le r<\frac{\alpha_0}{\Lambda},
$$
a contradiction.
\smallskip 

In case 
$$
D\cap (\p I\times\R) \ne\emptyset,
$$ 
we fix $\epsilon>0$ and choose an interval $J\strictlyincluded I$ with $|I\setminus J|<\epsilon$ and replace $D$ with $D_\epsilon:=D\cap ( J\times \R).$ Note that for small $\epsilon,$ $D_\epsilon$ is non-empty and satisfies $D_\epsilon\cap \Wulff_r^\norm (y) =\emptyset$ and $\sH^1(\p^*D_\epsilon\cap \p \Wulff_r^\norm(y))>0.$ Thus, we can use $E_u\setminus D_\epsilon$ as a competitor in \eqref{gstzfvb} to get \eqref{ahczs6gg} with $D_\epsilon$ in place of $D.$  Now letting $\epsilon\to0^+$ we conclude \eqref{ahczs6gg} and the remaining part of the contradictory argument runs as above. 
\smallskip

(b) is proven as (a).
\smallskip

(c) Fix any $x\in\Gamma_u,$  $r\in (0,\frac{\alpha_0}{\Lambda})$ and consider the set 
$$
\Sigma:=\{y\in\R^2:\,\,\dist_\norm(y,F_u)=\dist_\norm(y,\Gamma_u)=r\}.
$$ 
Note that $\Sigma $ contains two little arcs outside the strip $I\times\R,$ and hence, tangent balls may have centers outside it. Note that for any $y\in \Sigma$ the $\norm$-ball $\oWulff_r^\norm(y)$ is ``tangent'' to $\cl \Gamma_u$ at some point  and  does not intersect $F_u.$ In view of (a) and (b), it suffices    to show that there exists $\bar y\in\Sigma$ such that $x\in \p\Wulff_r^\norm(\bar y).$ Indeed, otherwise, as $\Gamma_u$ is a graph (an arcwise connected set), we could find $y\in \Sigma$ such that $ \Gamma_u \cap \p\Wulff_r^\norm(y)$ contains two distinct points $p$ and $q,$ and  $x$ lies in the relative interior of the subcurve of $\Gamma_u$ with endpoints $p$ and $q,$ but does not belong to $\p\Wulff_r^\norm(y).$ However, by (a), the set $\cl \Gamma_u \cap \p \Wulff_r^\norm(y)$ is connected, and hence $x\in \p\Wulff_r^\norm(y),$ a contradiction. 

For a similar reason, for any $x\in\Gamma_u$ and $r\in (0,\frac{\alpha_0}{\Lambda})$ there exists $\Wulff_r^\norm(y)$ such that $\oWulff_r^\norm(y)\cap E_u=\emptyset$ and $x\in \Gamma_u \cap \p\Wulff_r^\norm(y).$
\end{proof}

One corollary of Proposition  \ref{prop:uniform_wulff_locmin} 
is the following lipschitzianity of $u.$

\begin{corollary}\label{cor:lip_regular_funcos}
Let $\norm$ be  a partially monotone anisotropy such that $\Wulff^\norm$ does not have vertical facets and, 
given a bounded open interval $I\subset\R$ and $\Lambda>0,$ let $u\in BV_\loc(I)$ be a $\Lambda$-local minimizer of $\cA_\norm$ in $I$ satisfying \eqref{asdfvv}. Then $u$ is Lipschitz in $I.$
\end{corollary}

\begin{proof}
For simplicity, suppose  $I=(-a,a)$ for some $a>0.$ We claim that there exists $\lambda>0$ such that
\begin{equation}\label{normal_change}
\scalarp{\nu_{E_u} }{\be_2}\ge\lambda \quad\text{$\sH^1$-a.e. on $ (I\times\R)\cap \p^* E_u$},
\end{equation}
where $\nu_{E_u}$ is the generalized outer unit normal to $E_u.$ Indeed, by contradiction, assume that there exists a sequence $(x_k)\subset (I\times\R)\cap \p^* E_u$ such that $\scalarp{\nu_{E_u}(x_k)}{\be_2}\to0.$ Possibly passing to a not relabelled subsequence,  replacing $u$ with $-u$ and changing the orientation of $I$ (i.e., using the mirror symmetry with respect to the vertical axis) if necessary, we may assume $(x_k)\subset(-a,0]\times\R$ and $\nu_{E_u}(x_k)\cdot \be_1\to -1$ as $k\to+\infty.$
\begin{figure}[htp!]
\includegraphics[width=0.7\textwidth]{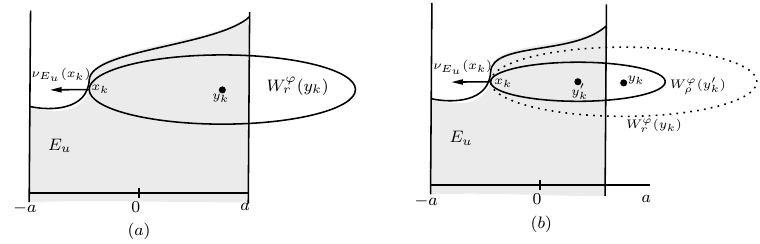}
\caption{\small }\label{fig:onct}
\end{figure}

Given $\epsilon\in(0,1),$ from Proposition \ref{prop:uniform_wulff_locmin} (c) 
select $r:=\frac{\alpha_0}{(1+\epsilon)\Lambda}$ and 
$\Wulff_r^\norm(y_k)$ so that $\oWulff_r^\norm(y_k)\cap F_u=\emptyset$ and $x_k\in \p\Gamma_u\cap \p\Wulff_r^\norm(y_k).$ 
Since $\nu_{E_u}(x_k)$ is an ``outer normal'' also to $\Wulff_r^\norm(y_k) $ at $x_k$ and $\Wulff^\norm$ has no vertical facets, it follows that 
$$
\scalarp{x_k-y_k}{\be_2} \to0.
$$ 
In particular, 
\begin{equation}\label{hstdzvb}
x_k-y_k = s_k \be_1 + t_k\be_2 
\quad \text{with $s_k\to \tfrac{r}{\norm(\be_1)}$ and $t_k\to0.$}
\end{equation}

First assume that, up to a not relabelled subsequence, $y_k\in (-a,a]\times\R$ for all $k\ge1,$ see Fig. \ref{fig:onct} (a). Then,  recalling that $\oWulff_r^\norm(y_k)\cap (I\times\R)\subset E_u,$ we have 
\begin{equation}\label{noski_bros_stirk}
\|u\|_\infty \ge 
\scalarp{y_k + \tfrac{r\be_2}{\norm(\be_2)}}{\be_2} = \scalarp{y_k}{\be_2} + \tfrac{r}{\norm(\be_2)} = \scalarp{x_k}{\be_2} + \tfrac{r}{\norm(\be_2)}-t_k,
\end{equation}
where $y_k + \tfrac{r\be_2}{\norm(\be_2)}$ is the top point of $\Wulff_r^\norm(y_k)$ in the vertical direction and in the last equality we used \eqref{hstdzvb}. As $x_k$ lies on the graph of $u,$
$
\scalarp{x_k}{\be_2} \ge -\|u\|_\infty,
$
and thus, from \eqref{noski_bros_stirk} and the definition of $r$ we deduce 
$$
2\|u\|_\infty \ge \tfrac{\alpha_0}{(1+\epsilon)\Lambda\norm(\be_2)}-t_k.
$$
Now first letting $k\to+\infty$ and then $\epsilon\to0^+$ we deduce 
$$
\|u\|_\infty \ge \tfrac{\alpha_0}{2\Lambda \norm(\be_2)},
$$
which contradicts  \eqref{asdfvv}.

Now assume that $y_k\in(a,+\infty)\times\R$ for all $k\ge1,$ see Fig. \ref{fig:onct} (b). In this case, by the partial monotonicity of $\norm,$ we have 
$
r>\rho:=a\norm(\be_1)
$
and hence $\rho<\frac{\alpha_0}{\Lambda}.$
Applying Proposition \ref{prop:uniform_wulff_locmin} (c) with $r=\rho$ we can find $\Wulff_\rho^\norm(y_k')$ such that $\oWulff_\rho^\norm(y_k')\cap F_u=\emptyset$ with $x_k\in \Gamma_u\cap \p \Wulff_\rho^\norm(y_k').$
As above, $\scalarp{x_k-y_k'}{\be_2} \to0,$ i.e., 
\begin{equation*}
x_k-y_k' = s_k' \be_1 + t_k'\be_2 
\quad \text{with $s_k'\to \tfrac{\rho}{\norm(\be_1)}$ and $t_k'\to0.$}
\end{equation*}
Since $\oWulff_\rho^\norm(y_k')\cap (I\times\R)\subset E_u,$ we can repeat the same arguments leading to \eqref{noski_bros_stirk} to get
$$
\|u\|_\infty \ge 
\scalarp{y_k' + \tfrac{\rho \be_2}{\norm(\be_2)}}{\be_2} = \scalarp{y_k'}{\be_2} + \tfrac{\rho}{\norm(\be_2)} = \scalarp{x_k}{\be_2} + \tfrac{\rho}{\norm(\be_2)}-t_k' \ge -\|u\|_\infty + \tfrac{\rho}{\norm(\be_2)}-t_k'. 
$$
Now letting $k\to+\infty$ and recalling the definition of $\rho$ we get 
$$
\|u\|_\infty \ge \tfrac{a\norm(\be_1)}{2\Lambda \norm(\be_2)} = 
\tfrac{|I|\norm(\be_1)}{4\Lambda \norm(\be_2)},
$$
which again contradicts   \eqref{asdfvv}.
\smallskip

These contradictions show the validity of \eqref{normal_change}.  In view of \eqref{normal_change} and \cite[Lemma 3.10]{NP:2002_narwa}, it follows that $\Gamma_u:=(I\times\R)\cap \p E_u$ is the graph of a Lipschitz function (with a Lipschitz constant $\sqrt{1/\lambda^2-1}$) in the vertical direction. Thus, $u$ admits a Lipschitz representative.
\end{proof}

Now we are ready to prove the regularity of $\Lambda$-local minimizers.

\begin{proof}[Proof of Theorem \ref{teo:jums_absent}]
By Corollary \ref{cor:lip_regular_funcos}, $u$ is Lipschitz in $I.$ 
Assume now $\norm$ is $C^2,$ elliptic and  partially monotone. Then so is $\norm^0.$ Moreover, the boundary of $\Wulff^\norm$ is a closed $C^2$-curve without segments. By Proposition  \ref{prop:uniform_wulff_locmin}, the subgraph $E_u:=\subgraph(u)$ of $u$ satisfies uniform\footnote{The radii of the tangent Wulff shapes can be choosen a constant $r\in(0,\frac{\alpha_0}{\Lambda})$ along the graph of $u.$} interior and exterior $\norm$-ball conditions at every point of $(I\times\R) \cap \p E_u.$  
In view of Proposition \ref{prop:elliptic_anis_propo} (c) this implies that $E_u$  and $F_u$ satisfies the classical ball condition of radius $\rho>0$ with a suitable $\rho>0$ depending only on $\alpha_0,$ $\Lambda$ and the constant $\bar r$ in Proposition \ref{prop:elliptic_anis_propo} (c). This allows us to obtain an $L^\infty$-bound for the second derivative of $u$ in terms of $1/\rho,$ which yields $u'$ is also Lipschitz in $I,$ see for instance   \cite[Section 2]{JN:2024_arma} for details. This and the lipschitzianity of $u$ imply $u\in C^{1,1}(I).$ 
\end{proof}

\subsection{Some generalizations}

In this section we relax the $\Lambda$-local minimality assumption on $u$ in Theorem \ref{teo:jums_absent}. To this aim, we start with the following 

\begin{definition}[$(\gamma,\Lambda)$-local minimizer]\label{def:cond_locmin}
Given an anisotropy $\norm$ in $\R^2,$ $\gamma>0,$  
$\Lambda\ge0,$ and a bounded open interval $I\subset\R,$ we say a function $u\in BV_\loc(I)\cap L^\infty(I)$ is a \emph{$(\gamma,\Lambda)$-local minimizer} provided that its subgraph $E_u:=\subgraph(u)$ satisfies 
\begin{equation}\label{dashstdfg}
\cP_\norm(E_u,\Omega) \le \cP_\norm(F,\Omega) + \Lambda|E_u\Delta F|  
\end{equation}
for any open set $\Omega\strictlyincluded I\times (-\gamma,\gamma)$ and $F\in BV_\loc(I\times\R;\{0,1\})$ with $E_u\Delta F\strictlyincluded \Omega.$
\end{definition}

Note that $(\gamma,\Lambda)$-local minimizers are not necessarily $\Lambda$-local minimizers, as local perturbations are taken only in $I\times (-\gamma,\gamma).$ Still, we can readily check that the density estimates in Lemma \ref{lem:density_estimates} and properties in \eqref{topo_boundo} hold, and therefore we can speak about closed and open representatives of $E_u.$ 
Moreover, in case $I$ and $u$ are bounded, we can apply \eqref{dashstdfg} with the set $D$ in the proof of Proposition  \ref{prop:uniform_wulff_locmin} provided for instance 
\begin{equation}\label{ming_dollasd}
\gamma > \frac{\alpha_0\norm(\be_1)}{2\Lambda};
\end{equation} 
for such $\gamma,$ if $\norm$ is partially monotone and $u$ satisfies \eqref{u_L_chek_baho}, all assertions  of Proposition  \ref{prop:uniform_wulff_locmin} are valid. This was sufficient to prove Theorem \ref{teo:jums_absent}. Thus, we have shown:

\begin{theorem}\label{teo:relaxed_jump}
Let $\norm$ be a partially monotone anisotropy in $\R^2$ such that $\Wulff^\norm$ does not have vertical facets,  $\gamma>0$ satisfy \eqref{ming_dollasd}, $\Lambda>0$ and $I\subset\R$ be a bounded interval. Then every $(\gamma,\Lambda)$-local minimizer $u$ of $\cA_\norm$ in  $I,$ satisfying the $L^\infty$-bound \eqref{asdfvv}, is Lipschitz in $I.$ If, additionally, $\norm$ is elliptic and $C^2,$ then $u\in C^{1,1}(I).$
\end{theorem}

\section{Applications}\label{sec:application}

\subsection{Minimizers of the perturbed area}

Let $\norm$ be an anisotropy in $\R^2,$  $I \subset\R$ be a bounded open interval and $g\in L^\infty(I).$ Given $p\ge1,$ consider the functional in \eqref{G_phi_pgI}, i.e., 
$$
\cG(u):=\cA_\norm(u,I) + \int_I|u-g|^pds,\quad u\in L^1(I),
$$
where we set $\cG(u)=+\infty$ if $u\notin BV(I)$ or $u-g\notin L^p(I).$ 

\begin{lemma}\label{lem:exist_unique}
There exists a minimizer $u\in L^1(I)$ of $\cG.$ Moreover, $u\in BV(I)\cap L^\infty(I)$ and $\|u\|_\infty \le \|g\|_\infty.$ Finally, if $p>1,$ then $u$ is unique.
\end{lemma}

Note that if $p=1,$ then in general minimizers are not unique, see  the Introduction.

\begin{proof}
The proof is standard and we provide it for completeness. Let $(u_k) \subset L^1(I)$ be a minimizing sequence. We may assume $\cG(u_k) \le \cG(0)$ for all $k,$ therefore, 
$$
\cA_\norm(u_k,I)\le \cG(0)\quad \text{and}\quad \int_I |u-g|^pds\le \cG(0).
$$
By the convexity of $\norm^o$ and \eqref{norm_bounds} we have 
$$
\cG(0)\ge \cA_\norm(u_k,I) \ge \cV_\norm(u_k,I) - \norm^o(\be_2)|I| \ge c \int_I |Du_k| - \norm^o(\be_2)|I|.
$$
Moreover, by the H\"older inequality,
$$
\int_I|u_k|ds \le \int_I|u_k-g|ds +\|g\|_\infty|I| \le  \Big(\int_I|u_k-g|^pds\Big)^{1/p} |I|^{1-\frac{1}{p}} \|g\|_\infty|I| \le (\cG(0))^{1/p}+\|g\|_\infty|I|.
$$
Thus, the sequence $(u_k)_k$ is bounded in $BV(I)$ and by the $L^1$-compactness in $BV,$ up to a not relabelled  subsequence, $u_k\to u$  in $L^1(I)$ for some $u\in BV(I).$ By the Riesz-Fischer lemma, we may also assume $u_k\to u$ a.e. in $I.$ Then by  the $L_\loc^1(I)$-lower semicontinuity of $\cA_\norm(\cdot,I),$
$$
\liminf_{k\to+\infty} \cA_\norm(u_k,I) \ge \cA_\norm(u,I)
$$
and by the Fatou's  lemma 
$$
\liminf_{k\to+\infty} \int_I|u_k-g|^pds \ge \int_I|u-g|^pds.  
$$
Thus, $u \in BV(I)$ is a minimizer of $\cG.$ 

To show that $\|u\|_\infty \le \|g\|_\infty,$ let
$
v:=\max\{u,-\|g\|_\infty\}.
$
Since $|u-g| \ge|v-g|$ a.e. in $I,$ we have 
$$
\int_I|u-g|^pds \ge \int_I|v-g|^pds
$$
with the strict inequality if the set $\{u<-\|g\|_\infty\}$ has positive measure.
Moreover, by Lemma \ref{lem:area_equal_perimeter} 
$$
\int_I\norm^o(-Du,1) = \cP_\norm(\subgraph(u),I\times\R)
=\cP_{\norm}(\epigraph(u),I\times\R).
$$ 
Since 
$$
\cP_{\norm}(\epigraph(v),I\times\R) = 
\cP_{\norm}(\epigraph(u) \cap [I\times (-\|g\|_\infty,+\infty)],I\times\R) \le 
\cP_{\norm}(\epigraph(u),I\times\R),
$$
where in the last inequality we used a cutting with half-spaces argument, see e.g. \cite{BKhN:2017_cpaa}, it follows that 
$$
\cA_\norm(u,I) = \cP_{\norm}(\epigraph(u),I\times\R) \ge \cP_{\norm}(\epigraph(v),I\times\R) \ge \cA_\norm(v,I).
$$
Thus $\cG(u)\ge \cG(v)$ with the strict inequality if the set $\{u<-\|g\|_\infty\}$ has positive measure. Then the minimality of $u$ implies $u\ge -\|g\|_\infty$ a.e. in $I.$ Similarly, we can show $u\le \|g\|_\infty$ a.e. in $I.$

The uniqueness of $u$ in case $p>1$ directly follows from the strict convexity of the $L^p$-norm.
\end{proof}

\begin{remark}\label{rem:truncation}
In fact, 
$$
-\|g^-\|_\infty \le -\|u^-\|_\infty \le \|u^+\|_\infty \le \|g^+\|_\infty,
$$
where $a^\pm=\max\{\pm a,0\}$ are the positive and negative parts of $a.$
\end{remark}

The next proposition establishes a bridge between minimizers of $\cG$ and $(\gamma,\Lambda)$-minimizers of $\cA_\norm.$

\begin{proposition}\label{prop:cond_locminimi}
Assume that $\norm$ is partially monotone and let $u\in BV(I)\cap L^\infty(I)$ be  a minimizer of $\cG.$ Then $u$ is a  $(\gamma,\Lambda)$-minimizer of $\cA_\norm$ in $I$ for any $\gamma>2\|g\|_\infty,$ where 
$$
\Lambda:=p(\gamma + \|g\|_\infty)^{p-1}.
$$
\end{proposition}

\begin{proof}
By Lemma \ref{lem:exist_unique}, $\|u\|_\infty\le \|g\|_\infty.$ Let $E_u:=\subgraph(u)$ be the subgraph of $u$ and consider any $F\in BV_\loc(I\times\R;\{0,1\})$ with $E_u\Delta F\strictlyincluded I\times (-\gamma,\gamma).$ Let $v$ be the vertical rearrangement of $F$ as in the proof of Proposition \ref{prop:locmin_func_set}. By construction, $v\in BV_\loc(I),$ $\supp(u-v)\strictlyincluded I$ and 
hence $v\in BV(I).$ In addition, by the Fubini-Tonelli theorem and the  partial monotonicity of $\norm,$
\begin{equation}\label{agctsfg}
|E_u\Delta F| = \int_I |u-v|ds\quad\text{and}\quad \cP_\norm(F,I\times\R) \ge \cP_\norm(\subgraph(v),I\times\R) = \cA_\norm(v,I),
\end{equation}
see for instance the proof of Proposition \ref{prop:locmin_func_set}. 
Moreover, since $\gamma>2\|u\|_\infty$ and $F$ does not cross the horizontal sides of the rectangle $I\times (-\gamma,\gamma),$ we have $\|v\|_\infty < \gamma.$ Thus, by the minimality of $u$ and \eqref{agctsfg},
\begin{multline}\label{uszc7bb}
\cP_\norm(E_u,I\times\R) = \cA_\norm(u,I) \le \cA_\norm(v,I) + \int_I\Big(|v-g|^p -|u-g|^p\Big)ds \\
\le \cP_\norm(F,I\times\R) + p\max\{|v-g|^{p-1},|u-g|^{p-1}\} \int_I|u-v|ds \le \cP_\norm(F,I\times\R) + \Lambda |E_u\Delta F|,
\end{multline}
where in the last inequality we used 
\begin{align*}
\max\{|v-g|^{p-1},|u-g|^{p-1}\}\le &  \max\{(\|v\|_\infty + \|g\|_\infty)^{p-1},(\|u\|_\infty + \|g\|_\infty)^{p-1}\} \\
\le & \max\{(\gamma + \|g\|_\infty)^{p-1},(2\|g\|_\infty)^{p-1}\} \le(\gamma + \|g\|_\infty)^{p-1}.
\end{align*}
Since $E_u\Delta F\strictlyincluded I\times(-\gamma,\gamma),$ comparing \eqref{uszc7bb} with \eqref{dashstdfg} we conclude that $u$ is a  $(\gamma,\Lambda)$-local minimizer of $\cA_\norm.$
\end{proof}

From Proposition \ref{prop:cond_locminimi} and Theorem \ref{teo:relaxed_jump} we deduce the following 

\begin{theorem}[Regularity of minimizers]\label{teo:regular_minimizers}
Let $\norm$ be a partially monotone anisotropy in 
$\mathbb R^2$ such that $\Wulff^\norm$ does not have vertical facets, $I\subset\R$ be a bounded open interval and $p\ge1.$ Let 
\begin{equation}\label{ashfgghh}
\sigma: = \Big( \tfrac{1}{4^{p-1}p}\,\min\Big\{\tfrac{\alpha_0 \norm(\be_1)}{4},\, \tfrac{\alpha_0} {2\norm(\be_2)},\, \tfrac{|I|\norm(\be_1)}{4\norm(\be_2)}\Big\}\Big)^{1/p},
\end{equation}
where $\alpha_0>0$ is defined in \eqref{def:alfa0}. Let $g\in L^\infty(I)$ be such that 
\begin{equation}\label{hstfbvznn}
\|g\|_\infty <\sigma.
\end{equation}
Then every minimizer $u$ of the functional $\cG$ in \eqref{G_phi_pgI} is Lipschitz in $I.$ Moreover, if $\norm$ is elliptic and $C^2,$ then $u\in C^{1,1}(I).$ 
\end{theorem}

\begin{proof}
Let $\gamma:=3\sigma$ and $\Lambda:=(\gamma+ \sigma)^{p-1}p = (4\sigma)^{p-1}p>0.$ By \eqref{ashfgghh},
$$
\sigma = \min\Big\{\tfrac{\alpha_0\norm(\be_1)}{4\Lambda},\,\tfrac{\alpha_0}{2\Lambda \norm(\be_2)},\,\tfrac{|I|\norm(\be_1)}{4\Lambda \norm(\be_2)}\Big\}.
$$
Thus, $\gamma$ satisfies \eqref{ming_dollasd}. Since $\|g\|_\infty < \sigma=\gamma/3,$ by Proposition \ref{prop:cond_locminimi}  $u$ is a  $(\gamma,\Lambda)$-local minimizer of $\cA_\norm$ in $I.$ Moreover, by Lemma \ref{lem:exist_unique}, $\|u\|_\infty\le \|g\|_\infty$ and hence, from \eqref{ashfgghh} and  \eqref{hstfbvznn} it follows that $u$ satisfies  \eqref{asdfvv}. Now the assertions directly follow from Theorem \ref{teo:relaxed_jump}.
\end{proof}

When $\norm$ is Euclidean and $p=1,$ \eqref{ashfgghh} reads as 
\begin{equation}\label{eq:explicit_sigma}
\sigma = \frac14\,\min\Big\{ \alpha_0, |I|\Big\},
\end{equation}
where $\alpha_0=\frac{2\sqrt{\pi}}{4\pi+1}.$ Thus, Theorem \ref{teo:regular_minimizers} implies that every minimizer of $\cG $ belongs to $C^{1,1}(I)$ provided that $\|g\|_\infty< \sigma.$ This positively answers to Conjecture \ref{conj:dg_area} in case $n=k=1,$ except for the dependence of $\sigma$ on $|I|.$ Note that our $\sigma$ depends on $|I|,$  while the $\sigma $ of \cite{CLSS:2025} (with $p=2$) depends only on $1/\sqrt{|I|}.$

The following example shows that in general the local $C^{1,1}$-regularity of $u$ in Theorem \ref{teo:regular_minimizers} cannot be improved.

\begin{example}\label{exa:C_1_1}
Let $I=(-1,1),$ $\norm$ be Euclidean and 
$$
g(s)=
\begin{cases}
a & \text{if $s\in(0,1),$}\\
-a & \text{if $s\in(-1,0)$}
\end{cases}
$$
for some $a\in(0,1).$ Then the $C^{1,1}(I)\setminus C^2(I)$-function (see Fig. \ref{fig:convex_hull} (a))
$$
u(s)= 
\begin{cases}
a & \text{if $s\in[\sqrt{2a-a^2},1),$} \\
a-1+\sqrt{1-(s-\sqrt{2a-a^2})^2} \quad & \text{if $s\in[0,\sqrt{2a-a^2}]$},\\
1-a-\sqrt{1-(s+\sqrt{2a-a^2})^2} & \text{if $s\in[-\sqrt{2a-a^2},0],$}\\
-a & \text{if $s\in(-1,-\sqrt{2a-a^2}]$}
\end{cases}
$$
is the unique minimizer of $\cG .$ Indeed, for simplicity, set 
$$
h(s):=\frac{u'(s)}{\sqrt{1+{u'(s)}^2}} =
\begin{cases}
0 & \text{if $s\in[\sqrt{2a-a^2},1]$}\\
-s+\sqrt{2a-a^2} \quad & \text{if $s\in[0,\sqrt{2a-a^2}]$}\\
s+\sqrt{2a-a^2} & \text{if $s\in[-\sqrt{2a-a^2},0]$}\\
0 & \text{if $s\in[-1,-\sqrt{2a-a^2}] $}
\end{cases}
$$
so that $h\in \Lip([-1,1])\cap C^\infty([-1,1]\setminus\{0,\pm\sqrt{2a-a^2}\})$ with 
\begin{equation}
\label{zd6cgv5}
h'=
\begin{cases}
0 & \text{a.e. in $\{u=g\},$}\\
-1 & \text{a.e. in  $\{u<g\},$}\\
1 & \text{a.e. in $\{u>g\}.$}
\end{cases}
\end{equation}
As $h(\pm1)=0,$ for any $v\in BV(-1,1)$ by integrating by parts we have 
\begin{multline*}
\int_{-1}^1 (u-v)h'ds = (u-v)h\big|_{-1}^1 - \int_{-1}^1 hu'ds+\int_{-1}^1 h\,dDv \\
=-\int_{-1}^1 \frac{{u'}^2ds}{\sqrt{1+{u'}^2}} + \int_{-1}^1 \frac{u'dDv}{\sqrt{1+{u'}^2}}
= -\int_{-1}^1 \sqrt{1+{u'}^2}ds  + \int_{-1}^1 \frac{ds + u'dDv}{\sqrt{1+{u'}^2}}.
\end{multline*}
On the other hand, by the explicit expression of $h',$
$$
\int_{-1}^1 (u-v)h'ds=\int_{-1}^1 (u-g)h'ds+\int_{-1}^1 (g-v)h'ds=\int_{-1}^1 |u-g|ds + \int_{-1}^1(g-v)h'ds.
$$
Combining these two equalities, we deduce
\begin{equation}\label{hstfgbn13rg}
\int_{-1}^1 \sqrt{1+{u'}^2}ds + \int_{-1}^1 |u-g|ds = \int_{-1}^1 \frac{ds+u'dDv}{\sqrt{1+{u'}^2}}+\int_{-1}^1(v-g)h'ds.
\end{equation}
By the H\"older inequality\footnote{If $\lambda$ and $\mu$ 
are bounded Radon measures in a bounded open set $I\subset\R^n$ and $p,q\in C(\cl{I}),$ there
holds
$$
\int_I pd\mu + qd\lambda \le \int_I\sqrt{p^2+q^2}\, d\sqrt{\mu^2+\lambda^2},
$$
where 
$$
\int_U \sqrt{\mu^2+\lambda^2}:=\sup\Big\{\int_U \phi d\mu + \psi d\lambda:\,\, (\phi,\psi)\in C_c(U;\R^2),\,\, \|\phi^2+\psi^2\|_\infty\le 1\Big\}
$$
is the total variation of $(\mu,\lambda)$ in the 
open set $U\subset I.$ We apply this  inequality with $p=(1+{u'}^2)^{-1/2},$ $q=u'(1+{u'}^2)^{-1/2},$ $\mu=\sL^1$ and $\lambda=Dv.$ 
} and the bound $\|h'\|_\infty\le1,$ 
$$
\int_{-1}^1 \frac{ds+u'dDv}{\sqrt{1+{u'}^2}} \le \int_{-1}^1\sqrt{1+|Dv|^2}\quad\text{and}\quad 
\int_{-1}^1(v-g)h'ds \le \int_{-1}^1|v-g|ds.
$$
Therefore, from \eqref{hstfgbn13rg} we deduce 
$\cG (u) \le \cG (v).$

Next, let us show the uniqueness of $u.$ 
\begin{figure}[htp!]
\includegraphics[width=0.7\textwidth]{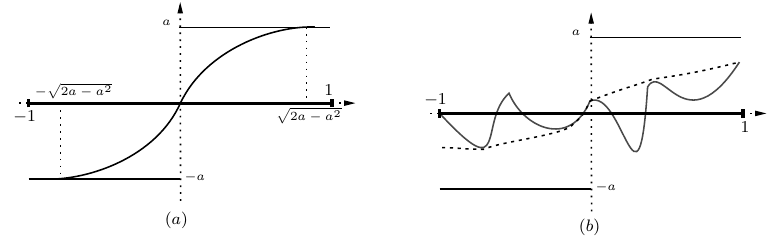}
\caption{\small  The graph of $u$ in (a) and convex/concave nondecreasing envelope of $v$ in (b).}\label{fig:convex_hull}
\end{figure}
Let $v\in BV(-1,1)$ be any other minimizer of $\cG .$ By Remark \ref{rem:truncation}, $-a\le v\le a.$ Let $v^*:[0,1]\to\R$ be the smallest nondecreasing concave function with $v^*\ge v$ a.e. in $[0,1]$ and $v_*$ be the largest nondecreasing convex function with $v_*\le v$ a.e. in $[0,1].$ Set $\bar v:=v^*\chi_{(0,1]}+v_*\chi_{[-1,0)}.$
As $g-v \ge g-\bar v\ge0$ in $(0,1)$ and
$v-g \ge \bar v-g\ge0$ in $(-1,0),$ we have 
$$
\int_{-1}^1 |g-v|ds \ge \int_{-1}^1 |g-\bar v|ds
$$
with strict inequality if $\{v\ne \bar v\}$ has   positive Lebesgue measure. Moreover, as we are replacing nonconcave/nonconvex parts of the graph of $v$ with line segments (see Fig. \ref{fig:convex_hull} (b)), 
$\cA_\norm(v,I) \ge \cA_\norm(\bar v,I).$ Thus, $\cG (\bar v)\le \cG (v).$ This inequality shows that we may assume $v=\bar v.$ In this case, by \eqref{relaxing}
$$
\int_{-1}^1 \frac{u'\,dDv}{\sqrt{1+{u'}^2}} = \int_{-1}^1\frac{u'v'ds}{\sqrt{1+{u'}^2}} + \sqrt{2a-a^2}\, (v^+(0)-v^-(0)).
$$
Thus, as above, from \eqref{hstfgbn13rg}, the H\"older inequality and \eqref{zd6cgv5} we get 
\begin{align}\label{dtdzg}
\cG (u)= & 
\int_{-1}^1 \sqrt{1+{u'}^2}ds + \int_{-1}^1 |u-g|ds = \int_{-1}^1 \frac{(1+u'v')ds}{\sqrt{1+{u'}^2}}+ \sqrt{2a-a^2}\, (v^+(0)-v^-(0)) +\int_{-1}^1|v-g|ds \nonumber\\
\le  & \int_{-1}^1 \sqrt{1+{v'}^2}ds + \sqrt{2a-a^2}\,(v^+(0)-v^-(0)) +\int_{-1}^1|v-g|ds \le \cG(v),
\end{align}
where in the last inequality we used $\sqrt{2a-a^2}<1.$ Since both $u$ and $v$ are minimizers, all inequalities in \eqref{dtdzg} are in fact equalities. In particular, $u'=v'$ a.e. in $(-1,1)$ and $v^+(0)=v^-(0).$ This implies $u=v+C$ for some real constant $C.$ Then,  recalling $u(\pm1)=\pm a$ and $-a\le v\le a,$ we deduce $C=0,$ i.e., $u=v.$
\end{example}

Notice that Example \ref{exa:C_1_1} shows that the threshold $\sigma$ in \eqref{eq:explicit_sigma} is not optimal, in general.

\subsection*{Data availability}

The paper has no associated data.

\subsection*{Author contributions and competing interests}

Both authors contributed equally to this paper. The authors declare that they have no competing interests.

\end{document}